\documentclass[11pt]{article}

\usepackage[letterpaper,top=2cm,bottom=2cm,left=3cm,right=3cm,marginparwidth=1.75cm]{geometry}
 
\usepackage{amsmath,amssymb,amsthm}
\usepackage{bbm}
\usepackage{enumerate}
\usepackage{comment}
\usepackage{algorithm}
\usepackage{algorithmic}
\usepackage{amsfonts}
\usepackage{commath}
\usepackage{enumitem}
\usepackage{mathtools}
\usepackage{nccmath}

\usepackage{graphicx}

\usepackage{pgfplots}
\pgfplotsset{compat=1.18}

\usepackage{pgfplotstable}
\usepackage{multirow}

\usepackage{etoolbox}

\usepackage{bm}

\usepackage{stmaryrd}
\usepackage{xparse}
\DeclarePairedDelimiterX{\Iintv}[1]{\llbracket}{\rrbracket}{\iintvargs{#1}}
\let\norm\relax
\DeclarePairedDelimiter\norm\lVert\rVert

\newcommand*\dt{{\Delta\kern-.8ptt}}

\DeclareMathOperator*{\argmin}{arg\,min}
\DeclareMathOperator*{\argmax}{arg\,max}

\usepackage{booktabs,dcolumn}
\newcolumntype{d}[1]{D{.}{.}{#1}} 

\newcommand{\bR}{\mathbb{R}}

\newcommand{\bN}{\mathbb{N}}

\newcommand{\dd}{\mathrm{d}}

\DeclareMathOperator{\prox}{\mathbf{prox}}

\usepackage{hyperref}
\hypersetup{
    colorlinks=true,
    linkcolor=blue,
    filecolor=magenta,
    citecolor =magenta,  
    urlcolor=magenta,
    pdftitle={Localized consensus-based sampling}
    }
\usepackage[capitalize]{cleveref}[0.19]

\crefname{section}{section}{sections}
\crefname{subsection}{subsection}{subsections}
\Crefname{section}{Section}{Sections}
\Crefname{subsection}{Subsection}{Subsections}

\Crefname{figure}{Figure}{Figures}

\crefformat{equation}{\textup{#2(#1)#3}}
\crefrangeformat{equation}{\textup{#3(#1)#4--#5(#2)#6}}
\crefmultiformat{equation}{\textup{#2(#1)#3}}{ and \textup{#2(#1)#3}}
{, \textup{#2(#1)#3}}{, and \textup{#2(#1)#3}}
\crefrangemultiformat{equation}{\textup{#3(#1)#4--#5(#2)#6}}%
{ and \textup{#3(#1)#4--#5(#2)#6}}{, \textup{#3(#1)#4--#5(#2)#6}}{, and \textup{#3(#1)#4--#5(#2)#6}}

\Crefformat{equation}{#2Equation~\textup{(#1)}#3}
\Crefrangeformat{equation}{Equations~\textup{#3(#1)#4--#5(#2)#6}}
\Crefmultiformat{equation}{Equations~\textup{#2(#1)#3}}{ and \textup{#2(#1)#3}}
{, \textup{#2(#1)#3}}{, and \textup{#2(#1)#3}}
\Crefrangemultiformat{equation}{Equations~\textup{#3(#1)#4--#5(#2)#6}}%
{ and \textup{#3(#1)#4--#5(#2)#6}}{, \textup{#3(#1)#4--#5(#2)#6}}{, and \textup{#3(#1)#4--#5(#2)#6}}

\newlist{lemenum}{enumerate}{1} 
\setlist[lemenum]{label=(\roman*), ref=\theproposition(\roman*), font=\rm}
\crefalias{lemenumi}{lemma} 

\newlist{propenum}{enumerate}{1} 
\setlist[propenum]{label=(\roman*), ref=\theproposition(\roman*), font=\rm}
\crefalias{propenumi}{proposition} 

\newtheorem{theorem}{Theorem}
\newtheorem{corollary}[theorem]{Corollary}
\newtheorem{definition}[theorem]{Definition}
\newtheorem{lemma}[theorem]{Lemma}
\newtheorem{remark}[theorem]{Remark}

\newlist{assumenum}{enumerate}{1} 
\setlist[assumenum]{leftmargin=2.1cm,label=(A\arabic*),font=\bfseries}
\creflabelformat{assumenumi}{#2#1#3}
\crefname{assumenumi}{assumption}{assumptions}
\Crefname{assumenumi}{Assumption}{Assumptions}

\title{A localized consensus-based sampling algorithm}
\author{Arne Bouillon\thanks{KU Leuven,
    Department of Computer Science (NUMA),
    Celestijnenlaan 200A, 3001 Leuven, Belgium, \href{mailto:arne.bouillon@kuleuven.be}{arne.bouillon@kuleuven.be}
    }\and Alexander Bodard\thanks{KU Leuven,
    Department of Electrical Engineering (ESAT-STADIUS),
    Kasteelpark Arenberg 10, 3001 Leuven, Belgium
}\and Panagiotis Patrinos\footnotemark[2]\and Dirk~Nuyens\footnotemark[1]~\and Giovanni Samaey\footnotemark[1]
}

\begin{document}
\maketitle

\begin{abstract}
    We propose a localized consensus-based method for sampling from non-Gaussian distributions, a task that frequently arises when solving Bayesian inverse problems. Our method arises from an alternative derivation of consensus-based sampling (CBS). Starting from ensemble-preconditioned Langevin dynamics, we replace the potential by its Moreau envelope---a smoother approximation---in order to replace the gradient in the Langevin equation with a proximal operator. We then approximate this operator by a weighted mean. In the limit of infinitely smoothing the potential to a quadratic function, this procedure recovers the standard CBS dynamics. In addition, outside this limit, we retrieve a refined variant of polarized CBS. We call the resulting algorithm \emph{localized consensus-based sampling}, since particles interact more with nearby particles than with faraway ones. Our method is affine-invariant, exact for Gaussian targets in the mean-field limit, and demonstrates improved robustness over polarized CBS in numerical experiments. Like other consensus-based methods, localized CBS is gradient-free and easily parallelizable.
\end{abstract}

\section{Introduction}
    We consider the problem of sampling a $d$-dimensional probability distribution with density\hspace{-.1cm}
    \begin{equation} \label{eq:intro:intro:dist}
        \pi(u) = \frac{\exp(-V(u))}{\int\exp(-V(v))\,\mathrm dv} \eqqcolon \frac{\widehat\pi(u)}Z,
    \end{equation}
    where $V\colon\bR^{d}\rightarrow\bR$ is a potential function with $\int\exp(-V(v))\mathrm dv<\infty$. Here, we implicitly take integrals over~$\bR^{d}$. In many applications $V$ is known but the integral~$Z$ may be very expensive to compute, such that $\pi$ is available only up to a \mbox{multiplicative constant}.

    \subsection{Bayesian inverse problems} \label{sec:intro:bip}
    \Cref{eq:intro:intro:dist} commonly arises in Bayesian inversion (see, e.g.,~\cite{kaipioStatisticalComputationalInverse2005,stuartInverseProblemsBayesian2010}). Consider a model $G$ with unknown parameters $u$ and perturbed observation\hspace{-.1cm}
    \begin{equation} \label{eq:intro:bip:bip}
        y = G(u) + \eta,
    \end{equation}
    where $\eta$ is some unknown measurement noise. The goal is characterizing likely values of~$u$, based on $y$ and $G$. Denoting by $\pi_\mathrm{prior}(u)$ and $\pi_\mathrm{noise}(\eta)$ our prior beliefs about $u$ and $\eta$, Bayes' rule describes our belief about $u$ after the observation by the \emph{posterior} distribution~with~density\hspace{-.1cm}
    \begin{equation} \label{eq:intro:bip:post}
        \pi_\mathrm{post}(u \mid y) = \frac{\pi_\mathrm{noise}(y - G(u))\,\pi_\mathrm{prior}(u)}{\int{\pi_\mathrm{noise}(y - G(v))\,\pi_\mathrm{prior}(v)}\,\mathrm dv} \eqqcolon \frac{\widehat\pi_\mathrm{post}(u\mid y)}{\pi_\mathrm{ev}(y)}.
    \end{equation}

    One approach to solve the Bayesian inverse problem~\cref{eq:intro:bip:bip} is to find the \emph{maximum a posteriori} (MAP) parameter $u_\mathrm{MAP}$ that maximizes~\cref{eq:intro:bip:post} as discussed in, e.g.,~\cite{dashtiMAPEstimatorsTheir2013}:
    \begin{equation}
        u_\mathrm{MAP} \coloneqq \argmax\nolimits_{u}\widehat\pi_\mathrm{post}(u\mid y).
    \end{equation}
    Optimization algorithms can solve the inverse problem in this fashion. However, $u_\mathrm{MAP}$ gives no information about the uncertainty on $u$ captured by $\pi_\mathrm{post}$. When uncertainty quantification is required, other characteristics of the posterior must be inferred. In that case, one commonly wants to integrate some function with respect to~$\pi_\mathrm{post}$, e.g., to compute statistical moments of the distribution. Numerical integration techniques, such as~\cite{herrmannQuasiMonteCarloBayesian2021}, which uses quasi-Monte Carlo, may provide a sufficient approximation even in high dimensions. Alternatively, it is common to use algorithms that produce a sequence of Monte Carlo samples distributed approximately according to~$\pi_\mathrm{post}$. Then solving~\cref{eq:intro:bip:bip} corresponds to drawing samples from~\cref{eq:intro:intro:dist}, with the \emph{evidence} $\pi_\mathrm{ev}(y)$ as the intractable $Z$ and with $\widehat\pi_\mathrm{post}(\cdot\mid y)$ as the unnormalized~\mbox{density}~$\widehat\pi$.

    \subsection{Related literature}
    Sampling from distributions with density~\cref{eq:intro:intro:dist} with unknown~$Z$ is historically done with Markov chain Monte Carlo (MCMC) algorithms, which originate in~\cite{hastingsMonteCarloSampling1970,metropolisEquationStateCalculations1953} and are reviewed in detail in~\cite{brooksHandbookMarkovChain2011}. MCMC remains a very active research topic with recent developments such as the popular pCN~\cite{cotterMCMCMethodsFunctions2013} and DILI~\cite{cuiDimensionindependentLikelihoodinformedMCMC2016} proposal moves that are effective in high dimensions and multilevel MCMC versions~\cite{dodwellHierarchicalMultilevelMarkov2015,lykkegaardMultilevelDelayedAcceptance2023} to exploit~cheap~\mbox{approximations}~to~$G$.

    An alternative method is sequential Monte Carlo (SMC)~\cite{delmoralSequentialMonteCarlo2006}, which evolves samples from an easy-to-sample distribution (e.g., $\pi_\mathrm{prior}$ in~\cref{eq:intro:bip:post}) into samples from $\pi$ gradually over multiple steps. Many other methods are based on evolving SDEs that have $\pi$ (or an approximation to $\pi$) as an invariant distribution. An example is Langevin diffusion (see \cref{sec:ips:lang}); after time discretization, the unadjusted Langevin algorithm (ULA)~\cite{robertsExponentialConvergenceLangevin1996} arises as an approximate sampler of $\pi$. Using ULA steps as MCMC proposals---or, equivalently, adding Metropolis--Hastings rejection to ULA---results in the Metropolis-adjusted Langevin algorithm~(MALA)~\cite{besagDiscussionPaperGrenander1994}.

    These Langevin methods use gradients of the potential $V$, which can be undesirable~\cite{dunbarEnsembleInferenceMethods2022b,kovachkiEnsembleKalmanInversion2019}. In the case of non-differentiable potentials, gradients can be approximated using the proximal operator~\cite{doi:10.1137/16M1108340,leeStructuredLogconcaveSampling2021,pereyraProximalMarkovChain2016,titsiasAuxiliaryGradientBasedSampling2018} or similar techniques~\cite{doi:10.1137/21M1406349} to form new, gradient-free algorithms. However, unless the potential has certain structure, this introduces the need to solve an embedded optimization problem. Various methods adapt Langevin dynamics to be gradient-free by using an \emph{ensemble} of multiple \emph{particles} (instances of the SDE) that exchange information, which can be harnessed to approximate gradients. These are often formulated in the infinite-particle limit, as McKean--Vlasov SDEs. Examples include ensemble Kalman sampling (EKS)~\cite{garbuno-inigoInteractingLangevinDiffusions2020} and the closely related ALDI method~\cite{garbuno-inigoAffineInvariantInteracting2020b}, derived from interacting Langevin diffusions and inspired by the ensemble Kalman inversion (EKI) optimizer~\mbox{\cite{iglesiasEnsembleKalmanMethods2013,kovachkiEnsembleKalmanInversion2019,schillingsAnalysisEnsembleKalman2017}}. EKI is related to iterating the ensemble Kalman filter (EnKF)~\cite{evensenSequentialDataAssimilation1994}; an EKI-based sampler is analyzed in~\cite{dingEnsembleKalmanInversion2021}. Other ensemble variants of Langevin dynamics include~\cite{dingConstrainedEnsembleLangevin2022} and the multiscale sampler in~\cite{pavliotisDerivativeFreeBayesianInversion2022}. In addition, the consensus-based sampling (CBS) method~\cite{carrilloConsensusbasedSampling2022}, described in \cref{sec:ips:cbs}, was developed independently of Langevin dynamics. It was inspired by consensus-based optimization~(CBO)~\cite{pinnauConsensusbasedModelGlobal2017}.

    Interaction in particle methods can be \emph{localized}: instead of computing a global interaction term, each particle interacts more with particles that are nearby than with those that are far away. The aim of localization is to avoid the linearity assumption on $G$ that is implicit in these methods. Localized algorithms show improved performance for nonlinear $G$, since the adapted dynamics rely on \emph{local} instead of global linearity. The paper~\cite{reichFokkerPlanckParticleSystems2021a} proposed to use localized interaction in the EKS/ALDI algorithm---details are given in \cref{sec:compare:laldi}. Similar ideas were proposed for EKI and the EnKF~\cite{wackerPerspectivesLocallyWeighted2024a}, for an ensemble Kalman method for rare-event sampling~\cite{wagnerEnsembleKalmanFilter2022c}, and for CBS as a variant named \emph{polarized CBS}~\cite{bungertPolarizedConsensusbasedDynamics2024}---see \cref{sec:compare:pcbs}. Our paper develops a different CBS variant with localized interaction, named \emph{localized CBS}. 

    \begin{remark}[Another concept of localization] \label{rem:intro:lit:loc}
        We stress that localization as discussed here is unrelated to the notion of localization in~\cite{houtekamerDataAssimilationUsing1998,ottLocalEnsembleKalman2004} for the EnKF, in~\cite{al-ghattasNonasymptoticAnalysisEnsemble2024,tongLocalizationEnsembleKalman2023} for EKI, and in~\cite{morzfeldLocalizationMCMCSampling2019} for MCMC. There, the empirical covariance matrix is corrected for spurious correlations, caused by a finite ensemble size, between parameter \emph{indices} that should be uncorrelated---e.g., because they correspond to elements at a large distance \emph{in the underlying problem}. This technique can also increase the rank of the covariance and prevent~EKI's~\emph{subspace~property}~\cite{tongLocalizationEnsembleKalman2023}. In contrast, we adopt the term \emph{localized} to denote localized interaction---where particles interact more with nearby particles than with faraway ones---following related literature such as~\cite{reichFokkerPlanckParticleSystems2021a,wagnerEnsembleKalmanFilter2022c}.
    \end{remark}

    \subsection{Objectives and contributions} \label{sec:intro:obj}
    There are several aspects to consider when assessing sampling algorithms, both computationally and in terms of accuracy.
    \begin{enumerate}[label={(\roman*)}]
        \item \emph{The need for gradients.} Methods that require gradients of $V$ cannot be used when gradients are unavailable.
        \item \emph{Parallelizability.} Standard MCMC methods are inherently sequential due to the need to overcome a burn-in phase and are hard to parallelize effectively~\cite{seelingerHighPerformanceUncertainty2021,yangParallelizableMarkovChain2018}. SMC and McKean--Vlasov-based methods use an ensemble of particles in parameter space that exchange information through interaction. This means that any work that is local to a single particle, such as evaluating $G$, can be trivially parallelized over all particles.
        \item \emph{Non-Gaussian performance.} Many sampling methods use approximations that hold only for Gaussian distributions, leading to poor sampling performance when the target distribution is non-Gaussian or multimodal. Non-Gaussian distributions frequently arise in practice, for instance when solving Bayesian inverse problems~\cref{eq:intro:bip:bip} with a nonlinear forward model $G$.
        \item \emph{Affine-invariance.} Some algorithms produce poor samples when the different dimensions of $\pi$ are poorly scaled. The property of affine-invariance (detailed in \cref{def:ips:lang:ai} in the following section) ensures that affine transformations of the parameters do not influence sampling performance.
    \end{enumerate}
    This paper focuses on gradient-free, parallelizable sampling from poorly scaled, non-Gaussian distributions. We make the following contributions: \begin{enumerate}[label={(\roman*)}]
        \item \Cref{sec:cbs-approx} proposes a novel way to connect interacting Langevin diffusions~\cite{garbuno-inigoInteractingLangevinDiffusions2020}, which use gradients, to consensus-based sampling~\cite{carrilloConsensusbasedSampling2022}, which does not. We rely on two approximations: replacing $V$ by a Moreau envelope and approximating the proximal operator by a weighted mean. CBS arises in the limit when the Moreau envelope reduces to a quadratic.
        \item \Cref{sec:lcbs:glob-prec,sec:lcbs:loc-prec} look at these approximations outside the CBS limit and tailor them to non-Gaussian problems. This results in an algorithm we name \emph{localized consensus-based sampling}. We perform an analysis of localized CBS with a broad class of covariance-based preconditioners. We derive ordinary differential equations (ODEs) for the evolution of its mean and covariance when applied to Gaussian problems in \cref{sec:lcbs:gauss}. More generally, \cref{sec:lcbs:ai} shows that localized CBS is affine-invariant.
        \item \Cref{sec:compare} compares the existing CBS, polarized CBS, and localized ALDI algorithms to our localized CBS method. We argue that our algorithm is an improved version of polarized CBS~\cite{bungertPolarizedConsensusbasedDynamics2024}, with our comparative benefits (such as affine-invariance and non-Gaussian performance) emerging from the connection to interacting Langevin diffusions.
        \item \Cref{sec:comput} discusses approximations of the mean-field equations with finite particle systems and other computational considerations. \Cref{sec:num} performs numerical experiments that showcase the method's affine-invariance and non-Gaussian performance.
    \end{enumerate}
    We start our discussion with a brief review of interacting Langevin diffusions and consensus-based sampling in \cref{sec:ips}.

    \subsection{Notation}
    We denote by $e_i$ the $i$th column of an identity matrix and by $\bm 1$ a vector of ones, both of implicit length. In addition, we define the Mahalanobis norm of a vector $v\in\bR^n$ weighted by the inverse of a positive definite matrix $A\in\bR^{n\times n}$ as
    \begin{equation}
        \norm{v}_A \coloneqq \sqrt{v^TA^{-1}\,v}.
    \end{equation}

\section{Interacting-particle samplers} \label{sec:ips}
    We now recall the interacting-Langevin and consensus-based samplers, and how they are used to sample in accordance with~\cref{eq:intro:intro:dist}.

    \subsection{Interacting-Langevin sampling} \label{sec:ips:lang}
    Let $W_t$ be a standard Brownian motion in $\bR^{d}$. The Langevin equation reads
    \begin{equation} \label{eq:ips:lang:lang}
        \dd U_t = \nabla\log\widehat\pi(U_t)\,\dd t + \sqrt2\,\dd W_t
    \end{equation}
    and, like all SDEs in this paper, should be interpreted in the It\^o sense. Precondition\-ing~\cref{eq:ips:lang:lang} with a positive definite matrix $K$ leads to a different SDE, the preconditioned Langevin equation
    \begin{equation} \label{eq:ips:lang:lang-K}
        \dd U_t = K\nabla\log\widehat\pi(U_t)\,\dd t + \sqrt{2K}\,\dd W_t.
    \end{equation}
    Denote by $\rho_t$ the law of $U_t$. Both~\cref{eq:ips:lang:lang} and~\cref{eq:ips:lang:lang-K} have the property that, under certain conditions on $\widehat\pi$, they transform arbitrary laws $\rho_0$ of the initial value $U_0$ into a law ${\rho_\infty \coloneqq \lim_{t\rightarrow\infty}\rho_t}$ that is exactly the desired distribution $\pi$. This leads to a Monte Carlo algorithm for approximately sampling from $\pi$: simulate some number~$J$ of discretized paths $\smash{\{U_t^j\}_{j=1}^J}$ of~\cref{eq:ips:lang:lang} or~\cref{eq:ips:lang:lang-K} until some sufficiently large time $T$ and use $\smash{\{U_T^j\}_{j=1}^J}$ as the approximate samples. If $K$ is chosen well,~\cref{eq:ips:lang:lang-K} can converge faster than~\cref{eq:ips:lang:lang}.

    In~\cite{garbuno-inigoInteractingLangevinDiffusions2020}, the constant matrix $K$ is replaced by a time-varying matrix $\mathcal C(\rho_t)$, where~$\rho_t$ is the law of $U_t$ and $\mathcal C$ denotes the covariance operator. This change results in
    \begin{equation} \label{eq:ips:lang:lang-int}
        \dd U_t = \mathcal C(\rho_t)\nabla\log\widehat\pi(U_t)\,\dd t+\sqrt{2\mathcal C(\rho_t)}\,\dd W_t.
    \end{equation}
    The covariance operator $\mathcal C$ uses the mean $\mu$ in its definition:
    \begin{equation}
        \mu(\rho) \coloneqq \medint\int v\,\mathrm d\rho(v) \qquad \text{and} \qquad \mathcal C(\rho)\coloneqq\medint\int(v - \mu(\rho))\otimes(v - \mu(\rho))\,\mathrm d\rho(v).
    \end{equation}
    \Cref{eq:ips:lang:lang-int} is not a special case of~\cref{eq:ips:lang:lang-K}; rather, it is an extension, and it avoids having to choose a suitable (problem-dependent) $K$. The corresponding Fokker--Planck equation is
    \begin{equation}
        \partial_t\rho_t(u) = \nabla \cdot \bigl(\rho_t(u)\,\mathcal C(\rho_t)\,(-\nabla\log\widehat\pi(u) + \nabla\log\rho(u))\bigr).
    \end{equation}
    It is shown in~\cite{garbuno-inigoInteractingLangevinDiffusions2020} that~\cref{eq:ips:lang:lang-int} evolves to $\rho_\infty=\pi$ at an exponential rate when $V$ is strongly convex and $\mathcal C(\rho_t)$ is bounded from below. In the non-convex case, $\pi$ is still an invariant distribution. In addition,~\cite{garbuno-inigoAffineInvariantInteracting2020b} shows that~\cref{eq:ips:lang:lang-int} has the important property of \emph{affine-invariance}.
    \begin{definition}[Affine-invariance] \label{def:ips:lang:ai}
        Consider an SDE that depends on the unnormalized target distribution $\widehat\pi$ (or, equivalently, on the potential $V$). The SDE is called affine-invariant if it is invariant under affine transformations of the state variables and $\widehat\pi$. That is, define
        \begin{equation}
            \tilde\pi(z) = \widehat\pi(Mz+b) \qquad \text{and} \qquad \tilde\rho_0(z) = \abs M\rho_0(Mz+b)
        \end{equation}
        for an arbitrary invertible matrix $M$ and vector $b$. Denote by $\rho_t$ the probability density at time~$t$ of solutions to the SDE for $\widehat\pi$ and with initial distribution $\rho_0$, and by $\tilde\rho_t$ the equivalent for $\tilde\pi$ and $\tilde\rho_0$. Then, the SDE is affine-invariant if
        \begin{equation}
            \tilde\rho_t(z) = \abs M\rho_t(Mz+b) \qquad \text{for all $t$}.
        \end{equation}
    \end{definition}
    Similar definitions are found in~\cite{garbuno-inigoAffineInvariantInteracting2020b,leimkuhlerEnsemblePreconditioningMarkov2018} for SDEs and, earlier, in~\cite{goodmanEnsembleSamplersAffine2010,greengardEnsemblizedMetropolizedLangevin2015} for discrete-time samplers. The importance of affine-invariance for samplers is widely acknowledged: it guarantees that the sampler will perform as well on poorly scaled distributions as on isotropic distributions. See~\cite{goodmanEnsembleSamplersAffine2010} for an extensive discussion and examples.

    Note that the SDE~\cref{eq:ips:lang:lang-int} is of McKean--Vlasov type\footnote{These McKean--Vlasov equations can be simulated by evolving an ensemble of many time-discretized solution paths in parallel and using their empirical distribution instead of the law~$\rho_t$. This is justified by the principle of \emph{propagation of chaos}~\cite{sznitmanTopicsPropagationChaos1991b}; for the ensemble Kalman sampler, this was studied in detail in~\cite{dingEnsembleKalmanSampler2021b}. Until \cref{sec:comput}, we study non-discretized McKean--Vlasov SDEs.}: its drift and diffusion depend on the law~$\rho_t$ through $\mathcal C(\rho_t)$. It is also possible to use any positive definite preconditioner $\mathbf P(\rho_t; U_t)$ that depends on the state $U_t$ as well as on its law~$\rho_t$. The resulting SDE is
    \begin{equation}
        \dd U_t = \mathbf P(\rho_t; U_t)\nabla\log\widehat\pi(U_t)\,\dd t+\sqrt{2\mathbf P(\rho_t; U_t)}\,\dd W_t
    \end{equation}
    and, similarly to what was noted in~\cite{nuskenNoteInteractingLangevin2019a,reichFokkerPlanckParticleSystems2021a}, has as its Fokker--Planck equation
    \begin{align*}
        \partial_t\rho_t(u) &= \nabla \cdot \left[-\rho_t(u)\,\mathbf P(\rho_t; u)\nabla\log\widehat\pi(u) + \mathbf P(\rho_t; u)\nabla\rho_t(u) + \rho_t(u)\nabla\cdot\mathbf P(\rho_t; u)\right]\\
        &= \nabla \cdot \left[\rho_t(u)\,\mathbf P(\rho_t; u)\,(-\nabla\log\widehat\pi(u) + \nabla\log\rho_t(u))\right] + \nabla \cdot \left[\rho_t(u) \nabla \cdot \mathbf P(\rho_t; u)\right].
    \end{align*}
    Without the last term, $\pi$ would again be invariant under this equation. Hence, one can introduce a correction term into the dynamics to cancel it out; this results in
    \begin{equation} \label{eq:ips:lang:lang-C}
        \dd U_t = \mathbf P(\rho_t; U_t)\nabla\log\widehat\pi(U_t)\,\dd t +\nabla_u\cdot\mathbf P(\rho_t; U_t)\,\dd t+\sqrt{2\mathbf P(\rho_t; U_t)}\,\dd W_t.
    \end{equation}
    Ideas related to~\cref{eq:ips:lang:lang-C} (with an added skew-symmetric matrix in the drift; see also~\cite{maCompleteRecipeStochastic2015}) have been discussed in~\cite{leimkuhlerEnsemblePreconditioningMarkov2018}. Furthermore,~\cref{eq:ips:lang:lang-C} underlies the finite-ensemble~\cite{nuskenNoteInteractingLangevin2019a} and localized~\cite{reichFokkerPlanckParticleSystems2021a} versions of EKS/ALDI through \emph{statistical linearization} (see~\cite{chadaIterativeEnsembleKalman2021}).

    \subsection{Consensus-based sampling} \label{sec:ips:cbs}
    Consensus-based sampling~\cite{carrilloConsensusbasedSampling2022} follows the dynamics\footnote{CBS also has an \emph{optimization mode} where the factor $(\alpha+1)$ is omitted, which results in an optimization method instead of a sampling one. This paper only considers CBS in sampling mode.}
    \begin{equation} \label{eq:cbs-approx:cbs}
        \dd U_t = -(U_t - \mu_{\alpha}(\rho_t))\,\dd t+\sqrt{2(\alpha+1)\,\mathcal C_{\alpha}(\rho_t)}\,\dd W_t.
    \end{equation}
    Here, $\alpha>0$ is a method parameter, $\rho_t$ is the law of $U_t$, and $W_t$ is a standard Brownian motion in $\bR^{d}$. The weighted mean and covariance $\mu_\alpha$ and $\mathcal C_\alpha$ are defined as
    \begin{subequations} \label{eq:cbs-approx:mu-and-C-alpha}
    \begin{align}
        \mu_{\alpha}(\rho)&\coloneqq\frac{\int vw_{\alpha}(v)\,\mathrm d\rho(v)}{\int w_{\alpha}(v)\,\mathrm d\rho(v)},\\
        \mathcal C_{\alpha}(\rho)&\coloneqq\frac{\int(v-\mu_{\alpha}(\rho))\otimes(v-\mu_{\alpha}(\rho))\,w_{\alpha}(v)\,\mathrm d\rho(v)}{\int w_{\alpha}(v)\,\mathrm d\rho(v)},
    \end{align}
    \end{subequations}
    where the weight function is given by
    \begin{equation}
        w_{\alpha}(v) \coloneqq \widehat\pi(v)^\alpha.
    \end{equation}
    The parameter $\alpha$ determines the importance of $\widehat\pi$ in weighting the particles; when $\alpha=0$, we get $(\mu_0, \mathcal C_0) = (\mu, \mathcal C)$. The SDE~\cref{eq:cbs-approx:cbs} is of McKean--Vlasov type, similarly to~\cref{eq:ips:lang:lang-C}. In~\cite{carrilloConsensusbasedSampling2022} it is shown that if~$\rho_0$ and~$\pi$~are Gaussian, then $\lim_{t\rightarrow\infty}\rho_t=\pi$ and the mean and covariance of $\rho_t$ converge exponentially in $t$. Similarly to interacting Langevin diffusions, CBS is affine-invariant. The mean-field limit of~\cref{eq:cbs-approx:cbs} was studied in~\cite{gerberMeanfieldLimitsConsensusBased2024} and~\cite{kossMeanFieldLimit2024}.

\section{Consensus-based sampling as approximate interacting Langevin} \label{sec:cbs-approx}
    We now propose a new connection between interacting-Langevin and consensus-based samplers, arriving at the CBS dynamics~\cref{eq:cbs-approx:cbs} starting from~\cref{eq:ips:lang:lang-C} by making three specific approximations. Throughout, we assume that the potential function $V$ is differentiable and $L$-weakly convex (meaning that $V(\cdot) + \frac L2\norm{\,\cdot\,}_2^2$ is convex for some $L\in\bR_0^+$) and that $\mathcal C_\alpha(\rho_t)$ is positive definite for all $t$.

    \paragraph{First step: Moreau envelope of $V$.} For a symmetric positive definite matrix $A$, denote by
    \begin{equation} \label{eq:cbs-approx:envelope}
        V^{A}(u) \coloneqq \inf_v\bigl(V(v) + \frac{1}{2}\norm{v - u}_A^2\bigr)
    \end{equation}
    the $A$-Moreau envelope of the potential $V$. It approximates $V$ from below and coincides at the minimizers, with $V^A(u) \nearrow V(u)$ for all $u$ as~$\norm{A}\searrow0$~\cite[\S 1G]{rockafellar_variational_1998}. If $V$ is $L$-weakly convex and $\norm{A}_2 < L^{-1}$, then $V^A$ is known to have gradient\hspace{-.1cm}
    \begin{equation} \label{eq:cbs-approx:env-grad}
        \nabla V^{A}(u) = A^{-1}(u - \prox^{A}_V(u))
    \end{equation}
    (see~\cite[Example 3.9]{Combettes01092014} and~\cite[Theorem 2.26]{rockafellar_variational_1998}), where the \emph{proximal operator} denotes the singleton
    \begin{equation}
        \prox^{A}_V(u) \coloneqq \argmin_v\bigl(V(v) + \frac{1}{2}\norm{v - u}_A^2\bigr).
    \end{equation}
    Let us now take~\cref{eq:ips:lang:lang-C} with~${\mathbf P(\rho_t; U_t) = \mathcal C_\alpha(\rho_t)}$ as a starting point. We introduce a scalar parameter~$\kappa > 0$ and approximate $V=-\log\widehat\pi$ in~\cref{eq:ips:lang:lang-C} by its $A$-Moreau envelope, with $A=\kappa\mathcal C_{\alpha}(\rho_t)$, rescaled by a scalar~$\gamma>0$: $V\approx\gamma V^{\kappa\mathcal C_{\alpha}(\rho_t)}$. We will later choose the parameter~$\gamma$ to counteract the approximation error of the Moreau envelope, with typically $\gamma\approx1$. As opposed to $\nabla V$, which is unknown, $\nabla(\gamma V^{\kappa\mathcal C_{\alpha}(\rho_t)})$ can be rewritten using~\cref{eq:cbs-approx:env-grad}, expressing our approximation to~\cref{eq:ips:lang:lang-C} as
    \begin{equation} \label{eq:cbs-approx:approx-1}
        \dd U_t = -\frac\gamma\kappa\left(U_t - \prox^{\kappa\mathcal C_{\alpha}(\rho_t)}_V(U_t)\right)\,\dd t + \sqrt{2\mathcal C_{\alpha}(\rho_t)}\,\dd W_t.
    \end{equation}
    Instead of a gradient, we now must evaluate the proximal operator.

    \paragraph{Second step: Weighted mean as approximate proximal operator.} Under mild \mbox{conditions} on a probability measure $\rho$ and function $f$, including $f$ having a unique minimizer, we have
    \begin{equation} \label{eq:lcbs:method:generic-laplace}
        \lim_{\beta\to\infty} \frac{\int v \exp(-\beta f(v))\,\mathrm d\rho(v)}{\int \exp(-\beta f(v))\,\mathrm d\rho(v)} = \argmin_v f(v),
    \end{equation}
    as implied by the Laplace principle~\cite{demboLargeDeviationsTechniques2010,millerAppliedAsymptoticAnalysis} (similarly to, e.g.,~\cite{carrilloAnalyticalFrameworkConsensusbased2018a,fornasierConsensusBasedOptimizationMethods2024}). Related is the \emph{quantitative Laplace principle} in~\cite{fornasierConsensusBasedOptimizationMethods2024}, which was used in conjunction with the proximal operator to study convergence of the consensus-based optimization (CBO) algorithm in~\cite{riedlGradientAllYou2023a}. \Cref{eq:lcbs:method:generic-laplace} implies
    \begin{equation} \label{eq:lcbs:method:laplace}
        \prox^{\kappa\mathcal C_{\alpha}(\rho_t)}_V(U_t) = \lim_{\beta\rightarrow\infty} \mu_{\beta,\kappa\mathcal C_{\alpha}(\rho_t)}(\rho_t; U_t),
    \end{equation}
    where we used the weighted mean
    \begin{equation} \label{eq:cbs-approx:w-and-mu-w}
        \mu_{\beta,A}(\rho; u) \coloneqq \frac{\int vw_{\beta,A}(v; u)\,\mathrm d\rho(v)}{\int w_{\beta,A}(v; u)\,\mathrm d\rho(v)}, \qquad w_{\beta,A}(v; u) \coloneqq \exp\Bigl(-\beta\Bigl[V(v) + \frac12\norm{v-u}_A^2\Bigr]\Bigr).
    \end{equation}
    As the smallest eigenvalue of $A$ tends to $\infty$, $\mu_{\beta,A}$ approaches $\mu_\beta$ from~\cref{eq:cbs-approx:mu-and-C-alpha}. The choice~$\rho=\rho_t$ in~\cref{eq:lcbs:method:laplace} was made to enable particle approximations in \cref{sec:comput}.

    A weighted mean was also used to approximate the proximal operator in~\cite{osherHamiltonJacobibasedProximalOperator2023} in the context of optimization methods. Motivated by~\cref{eq:lcbs:method:laplace}, we choose some finite~${\beta>0}$ and replace the proximal operator by a weighted mean as a second approximation:
    \begin{equation} \label{eq:cbs-approx:approx-2}
        \dd U_t = -\frac\gamma\kappa\Bigl(U_t - \mu_{\beta,\kappa\mathcal C_{\alpha}(\rho_t)}(\rho_t; U_t)\Bigr)\,\dd t + \sqrt{2\mathcal C_{\alpha}(\rho_t)}\,\dd W_t.
    \end{equation}
    Note that neither the proximal operator in~\cref{eq:cbs-approx:approx-1} nor the weighted mean in~\cref{eq:cbs-approx:approx-2} may seem inherently easier to compute than the gradient in~\cref{eq:ips:lang:lang-C}. The eventual reason for these approximations is that weighted means are suitable to efficient gradient-free particle discretizations, such as the ones in~\cite{carrilloConsensusbasedSampling2022} and \cref{sec:comput}. A direct particle discretization of~\cref{eq:ips:lang:lang-C} would clearly require gradients. For now we will study the dynamics in McKean--Vlasov form, but one should keep in mind this motivation.

    \paragraph{Recovering CBS in the limit.} If we take the limit $\kappa\to\infty$ in~\cref{eq:cbs-approx:approx-2} and allow $\gamma$ to depend on $\kappa$, then the localized part of the weighted mean vanishes and we get, with $c=\lim_{\kappa\to\infty}\gamma/\kappa$,
    \begin{equation}
        \dd U_t = -c\,\Bigl(U_t - \mu_{\beta}(\rho_t)\Bigr)\,\dd t + \sqrt{2\mathcal C_{\alpha}(\rho_t)}\,\dd W_t.
    \end{equation}
    Choosing $\alpha=\beta$ and $\gamma = (\alpha+1)^{-1}\kappa$ then gives exactly the CBS dynamics~\cref{eq:cbs-approx:cbs}.

    To interpret this link between~\cref{eq:cbs-approx:approx-2} and CBS, we study $V^\mathrm{CBS} \coloneqq \gamma V^{\kappa\mathcal C_{\alpha}(\rho_t)} = (\alpha+1)^{-1}\kappa V^{\kappa\mathcal C_{\alpha}(\rho_t)}$:%
    \begin{equation}
        V^\mathrm{CBS}(u) = (\alpha+1)^{-1}\inf_v\bigl(\kappa V(v) + \frac12\norm{v-u}_{\mathcal C_\alpha(\rho_t)}^2\bigr).
    \end{equation}
    As $\kappa\to\infty$, the infimum can only be reached by choosing $v\in\argmin V$, so
    \begin{equation}
        \lim_{\kappa\to\infty} V^\mathrm{CBS}(u) - \kappa\min V = (\alpha+1)^{-1}\,\frac12\operatorname{dist}_{\mathcal C_\alpha(\rho_t)}(u, \argmin V)^2,
    \end{equation}
    where $\operatorname{dist}_A(u, S) \coloneqq \inf_{v\in S}\norm{u-v}_A$. When $V$ has a unique minimizer $u^*$, this simplifies to
    \begin{equation}
        \lim_{\kappa\to\infty} V^\mathrm{CBS}(u) - \kappa\min V = (\alpha+1)^{-1}\,\frac12\norm{u-u^*}_{\mathcal C_\alpha(\rho_t)}^2.
    \end{equation}
    In other words, in the limit corresponding to CBS (i.e., $\kappa\to\infty$ with $\alpha=\beta$ and $\gamma = (\alpha+1)^{-1}\kappa$), the first approximation step in this section ``infinitely smooths'' the potential $V$ into a quadratic centered at the minimizer of $V$ for the purpose of approximating its gradient.

    \begin{remark}[(Weak) convexity of $V$]
        The above connection between Langevin dynamics and CBS requires convexity of $V$ such that $\nabla V^{\kappa\mathcal C_\alpha(\rho_t)}$ exists for any $\kappa$; the derivation leading up to the approximate Langevin dynamics~\cref{eq:cbs-approx:approx-2} also allows $V$ to be weakly convex with sufficiently small~$\kappa$.
        
        Weakly convex potentials are a natural setting for sampling methods. In addition, since the approximations in this section are only a motivation for CBS and~\cref{eq:cbs-approx:approx-2}, those dynamics might also be effective outside those assumptions. This should be investigated in future work.
    \end{remark}

\section{Localized consensus-based sampling} \label{sec:lcbs}
    Armed with the novel interpretation of CBS in \cref{sec:cbs-approx}, this section derives a new gradient-free sampling method. Our proposed localized consensus-based sampling algorithm retains affine-invariance, is derivative-free, and performs well on non-Gaussian distributions. We study localized CBS in time-continuous, mean-field form; \cref{sec:comput} will then consider discretizations.

    \subsection{Consensus-based sampling without Gaussian assumption} \label{sec:lcbs:glob-prec}
    Recall that our goal is a gradient-free sampling method with good performance on non-Gaussian distributions. We consider the approximations in \cref{sec:cbs-approx} that transform interacting Langevin diffusions to CBS, and adapt them to suit~this~new~\mbox{objective}.

    The $\kappa\to\infty$ limit of~\cref{eq:cbs-approx:approx-2}, leading to CBS, smooths $V$ to a quadratic. Since we want to sample from non-Gaussian $\pi$---i.e., non-quadratic $V$---we keep $\kappa$ finite. In addition, the Moreau envelope approximation~\cref{eq:cbs-approx:envelope} to $V$ suggests that $\kappa$ should be chosen small, while the weighted-mean approximation~\cref{eq:lcbs:method:laplace} to the proximal operator implies that $\beta$ should be large. Finally, we note that the steps leading up to~\cref{eq:cbs-approx:approx-2} are not specific to the preconditioner $\mathcal C_\alpha(\rho_t)$ and allow a general positive definite preconditioner $\mathbf P_t \coloneqq \mathbf P(\rho_t)$. This results in our novel localized CBS dynamics%
    \begin{equation} \label{eq:lcbs:glob-prec:lcbs}
        \dd U_t = -\frac\gamma\kappa\Bigl(U_t - \mu_{\beta,\kappa\mathbf P_t}(\rho_t; U_t)\Bigr)\,\dd t + \sqrt{2\mathbf P_t}\,\dd W_t.
    \end{equation}
    Here, $\kappa$ controls the accuracy of the Moreau envelope approximation (i.e., the amount of ``smoothing'' of $V$) and $\beta$ controls the accuracy of the weighted-mean approximation to the proximal operator (i.e., the accuracy of the gradient estimation). The preconditioner $\mathbf P_t$ can be chosen freely. Suitable choices include
    \begin{enumerate}[label={(\roman*{})}]
        \item a constant matrix $\mathbf P_t = K$;
        \item an unweighted covariance $\mathbf P_t = \mathcal C(\rho_t)$; and
        \item a weighted covariance $\mathbf P_t = \mathcal C_\alpha(\rho_t)$, which is used in classical CBS. Larger $\alpha$ puts more weight on particles with low potential values. This pushes $\mathbf P_t$ more towards the target covariance, rather than to the covariance of $\rho_t$, which may be beneficial when those differ significantly in their scaling.
    \end{enumerate}

    Finally, the parameter $\gamma$ could simply be set to $1$, but other choices may be able to partially mitigate the bias that is due to $\beta < \infty$~and~${\kappa > 0}$. We study this bias in more detail in \cref{sec:lcbs:gauss} as a function of $\gamma$, $\kappa$, and the choice of $\mathbf P_t$. Notably, when~$\pi$ is Gaussian and~$\mathbf P_t$ is an unweighted or weighted covariance---choices (ii) and (iii) above---the bias can be removed fully. The value of $\gamma$ that accomplishes this, given by~\cref{eq:cor:lcbs:gauss:gamma:gamma-C,eq:cor:lcbs:gauss:gamma:gamma-Calpha}, can then be used in~\cref{eq:lcbs:glob-prec:lcbs} even in non-Gaussian~cases~as~a~heuristic.

    \subsection{Localized preconditioners} \label{sec:lcbs:loc-prec}
    This subsection discusses more general preconditioners for the proposed localized CBS method, at the cost of a correction term and additional computational effort in the resulting algorithm. Concretely, we allow a preconditioner~$\mathbf P_{U,t}\coloneqq \mathbf P(\rho_t; U_t)$ that depends on the variable $U_t$ as well as its law~$\rho_t$\footnote{When the preconditioner can depend on $U_t$, a motivation similar to that in \cref{sec:cbs-approx} is still possible. In this case, the property~\cref{eq:cbs-approx:env-grad} gains a quadratic term. As $\norm{A}\rightarrow0$, this term vanishes, which justifies the use of~\cref{eq:cbs-approx:env-grad} as an approximation.}. As discussed in \cref{sec:ips:lang}, this requires an additional drift term $\nabla_u\cdot\mathbf P_{U,t}$:
    \begin{equation} \label{eq:lcbs:loc-prec:lcbs}
        \dd U_t = -\frac\gamma\kappa\Bigl(U_t - \mu_{\beta,\kappa\mathbf P_{U,t}}(\rho_t; U_t)\Bigr)\,\dd t + \nabla_u \cdot \mathbf P_{U,t} + \sqrt{2\mathbf P_{U,t}}\,\dd W_t.
    \end{equation}
    Computation of $\nabla_u\cdot\mathbf P_{U,t}$ is discussed in \cref{sec:comput}. The corresponding Fokker--Planck equation is
    \begin{equation} \label{eq:lcbs:method:fp}
        \partial_t\rho_t(u) = \nabla\cdot\Bigl(\frac\gamma\kappa\rho_t(u)\,\bigl(u-\mu_{\beta,\kappa\mathbf P_{U,t}}(\rho_t; u)\bigr) + \mathbf P_{U,t}\nabla\rho_t(u)\Bigr).
    \end{equation}
    A particular choice of interest for the space-dependent preconditioner is
    \begin{enumerate}[label={(iv)}]
        \item a weighted covariance $\mathbf P_{U,t} = \mathcal C_{\alpha, \alpha\lambda \mathcal C(\rho_t)}(\rho_t; U_t)$ with scalars $\alpha\ge0$ and $\lambda > 0$, where
        \begin{equation} \label{eq:lcbs:loc-prec:C-w}
            \mathcal C_{a,A}(\rho; u) \coloneqq \frac{\int (v - \mu_{a,A}(\rho; u))\otimes(v - \mu_{a,A}(\rho; u))\,w_{a,A}(v;u)\,\mathrm d\rho(v)}{\int w_{a,A}(v; u)\,\mathrm d\rho(v)},
        \end{equation}
    \end{enumerate}
    \Cref{eq:lcbs:loc-prec:C-w} uses the weighted mean and weight function~\cref{eq:cbs-approx:w-and-mu-w}. With this preconditioner, the bias of localized CBS can be removed for Gaussian $\pi$, similarly to the covariances in \cref{sec:lcbs:glob-prec}, with the $\gamma$ value given by~\cref{eq:cor:lcbs:gauss:gamma:gamma-Cw}. This weighted covariance is interesting for comparison to the \emph{polarized CBS} method discussed in \cref{sec:compare:pcbs}, since $\mathcal C_{\alpha, \alpha\lambda \mathcal C(\rho_t)}(\rho; u)$ corresponds to the polarized covariance~\cref{eq:compare:pcbs:stat-pol:C} with a Gaussian kernel\footnote{Polarized CBS does not include $\mathcal C(\rho_t)$ in the distance norm, so the correspondence is not perfect. We add this anisotropy in our covariance weights to ensure affine-invariance (see \cref{sec:lcbs:ai}).} $k(u, v) = \exp(-(2\lambda)^{-1}\norm{u-v}_{\mathcal C(\rho_t)}^2)$. Using such a weighted covariance can incur a significant extra computational cost and, in our experiments, does not seem necessary for good performance of localized CBS.

    \subsection{Analysis for Gaussian distributions} \label{sec:lcbs:gauss}
    While the localized CBS dynamics~\cref{eq:lcbs:loc-prec:lcbs} are designed to sample from non-Gaussian distributions, we will now analyze their behavior in the case where the target distribution is Gaussian, i.e., when~${\pi(u)\propto\widehat\pi(u) = \exp(-\frac12\norm{u - m}_\Sigma^2)}$ for some~$m$ and~$\Sigma$. This can be seen as verifying a minimum requirement: we can hardly expect the method to work well for non-Gaussian distributions if it cannot handle Gaussian ones. Moreover, this analysis results in values for $\gamma$ that we will later use even for non-Gaussian~$\pi$.

    \begin{lemma} \label{lmm:lcbs:gauss:moments}
        Assume that both the target distribution $\pi = \mathcal N(m, \Sigma)$ and the initial condition~${U_0 \sim \mathcal N(m_0, \Sigma_0)}$ are Gaussian. Then, the localized CBS dynamics~\cref{eq:lcbs:loc-prec:lcbs} admit a Gaussian solution $U_t\sim\mathcal N(m_t, \Sigma_t)$ with mean and covariance that satisfy
        \begin{subequations} \label{eq:lmm:lcbs:gauss:moments}
        \begin{align}
            \frac\dd{\dd t} m_t &=-\frac{\gamma\beta}\kappa Q_t\Sigma^{-1}(m_t - m), \label{eq:lmm:lcbs:gauss:moments:evol-m}\\
            \frac\dd{\dd t}\Sigma_t &= 2\Bigl[\mathbf P_{*,t} - \frac\gamma\kappa\Sigma_t + \frac{\gamma\beta}{2\kappa^2}\left(Q_t\mathbf P_{*,t}^{-1}\Sigma_t + \Sigma_t\mathbf P_{*,t}^{-1}Q_t\right)\Bigr], \label{eq:lmm:lcbs:gauss:moments:evol-C}
        \end{align}
        \end{subequations}
        with $Q_t \coloneqq (\beta\Sigma^{-1} + \Sigma_t^{-1} + \frac\beta\kappa\mathbf P_{*,t}^{-1})^{-1}$, if the preconditioner $\mathbf P_{*,t}\coloneqq \mathbf P(\rho_t; U_t)$ does not depend on $U_t$ when $\rho_t = \mathcal N(m_t, \Sigma_t)$. This is the case for the preconditioners considered before:
        \begin{subequations} \label{eq:lmm:lcbs:gauss:moments:C}
        \begin{align}
            \mathbf P(\rho_t; U_t) &= K, \label{eq:lmm:lcbs:gauss:moments:C:K}\\
            \mathbf P(\rho_t; U_t) &= \mathcal C(\rho_t) = \Sigma_t, \label{eq:lmm:lcbs:gauss:moments:C:C}\\
            \mathbf P(\rho_t; U_t) &= \mathcal C_\alpha(\rho_t) = (\alpha\Sigma^{-1}+\Sigma_t^{-1})^{-1}, \label{eq:lmm:lcbs:gauss:moments:C:Calpha}\\
            \mathbf P(\rho_t; U_t) &= \mathcal C_{\alpha, \alpha\lambda\mathcal C(\rho_t)}(\rho_t; U_t) = (\alpha\Sigma^{-1} + (1 + \lambda^{-1})\Sigma_t^{-1})^{-1}, \label{eq:lmm:lcbs:gauss:moments:C:Cw}
        \end{align}
        \end{subequations}
        with parameters $\alpha\ge 0$, $\lambda>0$, and $K\succ0$.
    \end{lemma}
    \begin{proof}
        Assume $\rho_t = \mathcal N(m_t, \Sigma_t)$. From the condition that $\mathbf P_{*,t}$ not depend on~$U_t$ follows that~${\nabla_u\cdot\mathbf P_{*,t} = 0}$. \Cref{sec:apdx} derives the rightmost terms in~\cref{eq:lmm:lcbs:gauss:moments:C:C,eq:lmm:lcbs:gauss:moments:C:Calpha,eq:lmm:lcbs:gauss:moments:C:Cw} and
        \begin{equation} \label{eq:lmm:lcbs:gauss:moments:mu-w}
            \mu_{\beta, \kappa\mathbf P_{*,t}}(\rho_t; U_t) = (\beta\Sigma^{-1} + \Sigma_t^{-1} + \frac\beta\kappa\mathbf P_{*,t}^{-1})^{-1}(\beta\Sigma^{-1}m + \Sigma_t^{-1}m_t + \frac\beta\kappa\mathbf P_{*,t}^{-1}U_t).
        \end{equation}
        In this case,~\cref{eq:lcbs:loc-prec:lcbs} becomes a \emph{linear SDE}: the drift term depends on $U_t$ in an affine way and the diffusion term is $U_t$-independent. Following, e.g.,~\cite[Section 6.1]{sarkkaAppliedStochasticDifferential2019}, the SDE admits a Gaussian solution with mean and covariance that satisfy
        \begin{align}
        &\begin{aligned}
            \frac\dd{\dd t} m_t &= -\frac\gamma\kappa(I - \frac\beta\kappa Q_t\mathbf P_{*,t}^{-1}) m_t + \frac\gamma\kappa Q_t\,(\beta\Sigma^{-1}m + \Sigma_t^{-1}m_t)\\
            &= -\frac\gamma\kappa\Bigl[(I - Q_t\,(Q_t^{-1} - \beta\Sigma^{-1}))m_t - Q_t\beta\Sigma^{-1}m\Bigr] = -\frac{\gamma\beta}\kappa Q_t\Sigma^{-1}\,(m_t - m),
        \end{aligned}\\
            &\kern1.5pt\frac\dd{\dd t} \Sigma_t = -\frac\gamma\kappa(I - \frac\beta\kappa Q_t\mathbf P_{*,t}^{-1})\Sigma_t -\frac\gamma\kappa\Sigma_t\,(I - \frac\beta\kappa \mathbf P_{*,t}^{-1}Q_t) + 2\mathbf P_{*,t}.
        \end{align}
    \end{proof}

    \begin{corollary} \label{cor:lcbs:gauss:gamma}
        Assume that the target distribution $\pi=\mathcal N(m, \Sigma)$ is Gaussian. The localized CBS dynamics admit $\pi$ as a stationary distribution in the following~cases.
        
        \begin{enumerate}[label={(\roman*{})}]
            \item The preconditioner is the unweighted covariance $\mathbf P_{U,t} = \mathcal C(\rho_t)$ and
            \begin{equation} \label{eq:cor:lcbs:gauss:gamma:gamma-C}
                \gamma = \kappa + \frac\beta{\beta+1}.
            \end{equation}
            \item The preconditioner is the weighted covariance $\mathbf P_{U,t} = \mathcal C_\alpha(\rho_t)$ with $\alpha\ge0$ and
            \begin{equation} \label{eq:cor:lcbs:gauss:gamma:gamma-Calpha}
                \gamma = (\alpha+1)^{-1}\kappa + \frac\beta{\beta+1}.
            \end{equation}
            \item The preconditioner is the weighted covariance $\mathbf P_{U,t} = \mathcal C_{\alpha, \alpha\lambda\mathcal C(\rho_t)}(\rho_t; U_t)$ with $\alpha\ge0$, $\lambda>0$, and
            \begin{equation} \label{eq:cor:lcbs:gauss:gamma:gamma-Cw}
                \gamma = \Bigl(\frac1\lambda + \alpha+1\Bigr)^{-1}\kappa + \frac\beta{\beta + 1}.
            \end{equation}
        \end{enumerate}
    \end{corollary}
    \begin{proof}
        If $\pi$ is a stationary distribution of localized CBS, then ${\partial_t\rho_t(u) = 0}$ when~${\rho_t = \pi}$. Since then both $\rho_t$ and $\pi$ are Gaussian, we can use \cref{lmm:lcbs:gauss:moments}, fill in an expression for the preconditioner such as~\cref{eq:lmm:lcbs:gauss:moments:C}, and require $\frac\dd{\dd t} m_t = 0$ and $\frac\dd{\dd t} \Sigma_t = 0$ when $m_t=m$ and $\Sigma_t = \Sigma$. Clearly, $\frac\dd{\dd t} m_t$ as given by~\cref{eq:lmm:lcbs:gauss:moments:evol-m} is zero irrespective of $\gamma$ and $\mathbf P_{*,t}$ when $m_t=m$. We also have
        \begin{equation} \label{eq:cor:lcbs:gauss:gamma:Sigmadot-1}
        \begin{aligned}
            \frac\dd{\dd t}\Sigma_t &= 2\Bigl[\mathbf P_{*,t} - \frac\gamma\kappa\Sigma + \frac{\gamma\beta}{2\kappa^2}\left(Q_t\mathbf P_{*,t}^{-1}\Sigma + \Sigma \mathbf P_{*,t}^{-1}Q_t\right)\Bigr].
        \end{aligned}
        \end{equation}
        When $\mathbf P_{*,t}$ and $\Sigma_t = \Sigma$ commute---which is the case for the three preconditioners considered in this corollary, as seen in~\cref{eq:lmm:lcbs:gauss:moments:C}---we can simplify~\cref{eq:cor:lcbs:gauss:gamma:Sigmadot-1} to
        \begin{equation}
        \begin{aligned}
            \frac\dd{\dd t}\Sigma_t &= 2\Bigl[\mathbf P_{*,t} + \frac\gamma\kappa\Bigl(\bigl(\kappa\frac{\beta+1}\beta\mathbf P_{*,t}\Sigma^{-1} + I\bigr)^{-1} - I\Bigr)\Sigma\Bigr]\\
            &= 2\Bigl[\mathbf P_{*,t} - \frac\gamma\kappa\bigl(\kappa\frac{\beta+1}\beta\mathbf P_{*,t}\Sigma^{-1} + I\bigr)^{-1}\bigl(\kappa\frac{\beta+1}\beta\mathbf P_{*,t}\Sigma^{-1}\bigr)\Sigma\Bigr]\\
            &= 2\Bigl[\mathbf P_{*,t} - \gamma\,\bigl(\kappa\Sigma^{-1} + \frac\beta{\beta+1}\mathbf P_{*,t}^{-1}\bigr)^{-1}\Bigr],
        \end{aligned}
        \end{equation}
        which, after some algebraic manipulations, can be seen to equal zero when
        \begin{equation} \label{eq:cor:lcbs:gauss:gamma:cond}
            \kappa I = \Bigl(\gamma - \frac\beta{\beta+1}\Bigr)\Sigma\mathbf P_{*,t}^{-1}.
        \end{equation}
    \Cref{eq:lmm:lcbs:gauss:moments:C:Cw} becomes $\mathbf P_{*,t} = (\lambda^{-1} + \alpha+1)\Sigma$ when $\Sigma_t = \Sigma$, which means that~\cref{eq:cor:lcbs:gauss:gamma:cond} holds when $\gamma$ is given by~\cref{eq:cor:lcbs:gauss:gamma:gamma-Cw}. The other cases correspond to specific choices for $\lambda$ and $\alpha$.
    \end{proof}
    A statement such as \cref{cor:lcbs:gauss:gamma} about the preconditioner $\mathbf P_{*,t} = K$ is not possible, since there does not exist a $\gamma$ for which~\cref{eq:cor:lcbs:gauss:gamma:Sigmadot-1} is satisfied for all $\Sigma$ in this case.

    \begin{remark}[Non-Gaussian performance] \label{rem:lcbs:gauss:nongauss}
        In the non-Gaussian case, localized CBS with finite parameters $(\beta, \kappa)$ will not sample from the exact target distribution $\pi$ in general. Characterizing the well-posedness of~\cref{eq:lcbs:loc-prec:lcbs} and the relation between $\pi$ and a possible stationary distribution $\rho_\infty$ of localized CBS is an open question. 
    \end{remark}

    \subsection{Affine-invariance} \label{sec:lcbs:ai}
    We now study affine-invariance of localized CBS.

    \begin{lemma}
        The SDE~\cref{eq:lcbs:loc-prec:lcbs} is affine-invariant (see \cref{def:ips:lang:ai}) when the preconditioner is given by $\mathbf P(\rho_t; U_t) = \mathcal C(\rho_t)$, $\mathbf P(\rho_t; U_t) = \mathcal C_\alpha(\rho_t)$, or $\mathbf P(\rho_t; U_t) = \mathcal C_{\alpha, \alpha\lambda \mathcal C(\rho_t)}(\rho_t; U_t)$.
    \end{lemma}
    \begin{proof}
        To prove affine-invariance following \cref{def:ips:lang:ai}, we use the Fokker--Planck equation~\cref{eq:lcbs:method:fp} and make the dependence on $\widehat\pi$ explicit:
        \begin{equation}
            \partial_t\rho_t(u) = \nabla\cdot\Bigl(\frac\gamma\kappa\rho_t(u)\,\bigl(u-\mu_{\beta,\kappa\mathbf P^{\widehat\pi}(\rho_t; u)}^{\widehat\pi}(\rho_t; u)\bigr) + \mathbf P^{\widehat\pi}(\rho_t; u)\nabla\rho_t(u)\Bigr).
        \end{equation}
        With the transformed $\tilde\pi(z) = \widehat\pi(Mz+b)$, the Fokker--Planck equation looks like
        \begin{equation} \label{eq:lmm:lcbs:ai:ai:fp-trans}
            \partial_t\tilde\rho_t(z) = \nabla\cdot\Bigl(\frac\gamma\kappa\tilde\rho_t(z)\,\bigl(z-\mu_{\beta,\kappa\mathbf P^{\tilde\pi}(\tilde\rho_t; z)}^{\tilde\pi}(\tilde\rho_t; z)\bigr) + \mathbf P^{\tilde\pi}(\tilde\rho_t; z)\nabla\tilde\rho_t(z)\Bigr).
        \end{equation}
        Now define
        \begin{equation}
            \bar\rho_t(z) = \abs M\rho_t(Mz+b);
        \end{equation}
        then, by \cref{def:ips:lang:ai}, localized CBS is affine-invariant if~\cref{eq:lmm:lcbs:ai:ai:fp-trans} holds for $\tilde\rho_t = \bar\rho_t$. As a first step, one can check that these preconditioners satisfy (with $u=Mz+b$)
        \begin{equation}
            \mathbf P^{\widehat\pi}(\rho_t; u) = M\mathbf P^{\tilde\pi}(\bar\rho_t; z)M^T
        \end{equation}
        and, thus, that
        \begin{equation}
            \mu_{\beta,\kappa\mathbf P^{\widehat\pi}(\rho_t; u)}^{\widehat\pi}(\rho_t; u) = M\mu_{\beta,\kappa\mathbf P^{\tilde\pi}(\bar\rho_t; z)}^{\tilde\pi}(\bar\rho_t; z) + b.
        \end{equation}
        We can then write
        \begin{align*}
            \partial_t\bar\rho_t(z) &= \abs M\partial_t\rho_t(u)\\
            &= \abs M\nabla_u\cdot\Bigl(\frac\gamma\kappa\rho_t(u)\,(u-\mu_{\beta,\kappa\mathbf P^{\widehat\pi}(\rho_t; u)}^{\widehat\pi}(\rho_t; u)) + \mathbf P^{\widehat\pi}(\rho_t; u)\nabla_u\rho_t(u)\Bigr)\\
            &= \nabla_u\cdot\Bigl(\frac\gamma\kappa M\bar\rho_t(z)\,(z - \mu_{\beta,\kappa\mathbf P^{\tilde\pi}(\bar\rho_t; z)}^{\tilde\pi}(\bar\rho_t; z)) + M\mathbf P^{\tilde\pi}(\bar\rho_t; z)\nabla_z\bar\rho_t(z)\Bigr)\\
            &= \nabla_z\cdot\Bigl(\frac\gamma\kappa\bar\rho_t(z)\,(z - \mu_{\beta,\kappa\mathbf P^{\tilde\pi}(\bar\rho_t; z)}^{\tilde\pi}(\bar\rho_t; z)) + \mathbf P^{\tilde\pi}(\bar\rho_t; z)\nabla_z\bar\rho_t(z)\Bigr).
        \end{align*}
        This proves affine-invariance.
    \end{proof}

\section{Comparison to alternative methods} \label{sec:compare}
    This section compares the dynamics~\cref{eq:lcbs:loc-prec:lcbs} of localized CBS to those of various other McKean--Vlasov-based samplers. \Cref{sec:compare:cbs} discusses classical CBS, after which \cref{sec:compare:pcbs,sec:compare:laldi} compares our method to two algorithms that use localization in a similar way. Various other dynamics are briefly discussed in \cref{sec:compare:other}. A detailed numerical comparison to some of these methods follows in \cref{sec:num}; here, we focus on different aspects of the dynamics and how these may affect the properties of the resulting algorithms.

    \subsection{Consensus-based sampling} \label{sec:compare:cbs}
    We derived localized CBS in \cref{sec:lcbs} as a generalization of classical CBS to non-Gaussian sampling problems. We will now show that, when both the target distribution $\pi$ and the initial particle distribution $\rho_0$ are Gaussian, the localized CBS dynamics~\cref{eq:lcbs:loc-prec:lcbs} with a certain weighted covariance as preconditioner admit the same solution as the CBS dynamics~\cref{eq:cbs-approx:cbs}.

    \begin{lemma} \label{lmm:compare:cbs:equiv}
        Let $U_t \sim \mathcal N(m_t, \Sigma_t)$ follow the localized CBS dynamics~\cref{eq:lcbs:loc-prec:lcbs} preconditioned with the weighted covariance $\mathbf P(\rho; u) = \mathcal C_{\alpha, \alpha\lambda\mathcal C(\rho)}(\rho; u)$ and let $U_t' \sim \mathcal N(m_t', \Sigma_t')$ follow the CBS dynamics~\cref{eq:cbs-approx:cbs}. Let both the target density $\pi = \mathcal N(m, \Sigma)$ and the initial densities of both dynamics ${\mathcal N(m_0, \Sigma_0) = \mathcal N(m_0', \Sigma_0')}$ be Gaussian.

        For any choice of the CBS parameter\footnote{We write $\alpha'$ for the CBS parameter to distinguish it from the $\alpha$ parameter of localized CBS.} $\alpha' > 0$ and with localized CBS parameters
        \begin{equation} \label{eq:lmm:compare:cbs:equiv:params}
            \alpha > 0, \qquad \kappa = \alpha, \qquad \beta = \alpha', \qquad \text{and} \qquad \lambda = \frac{\alpha'}{\alpha - \alpha'},
        \end{equation}
        it holds that $m_t = m'_{ct}$ and $\Sigma_t = \Sigma'_{ct}$ with
        \begin{equation}
            c = \frac{\alpha'}{\alpha\,(\alpha'+1)}.
        \end{equation}
    \end{lemma}
    \begin{proof}
        In~\cite{carrilloConsensusbasedSampling2022}, it is shown that
        \begin{equation} \label{eq:lmm:compare:cbs:equiv:dot-Sigma'}
            \frac\dd{\dd t}\Sigma'_t = -2\alpha'\Sigma'_t\,(\Sigma + \alpha'\Sigma'_t)^{-1} (\Sigma'_t - \Sigma).
        \end{equation}
        On the other hand, filling in~\cref{eq:lmm:compare:cbs:equiv:params} into~\cref{eq:lmm:lcbs:gauss:moments:evol-C} with~\cref{eq:lmm:lcbs:gauss:moments:C:Cw} results in
        \begin{align*}
            \frac\dd{\dd t}\Sigma_t &= 2\Bigl[\mathbf P - 2\frac{\alpha'}{\alpha(\alpha'+1)}\Sigma_t + \frac{(\alpha')^2}{\alpha^2(\alpha'+1)}\left(Q_t\mathbf P^{-1}\Sigma_t + \Sigma_t\mathbf P^{-1}Q_t\right)\Bigr]\\
            &= 2\Bigl[\mathbf P - 2c\Sigma_t + c\frac{\alpha'}\alpha\bigl(\frac\alpha{2\alpha'}I \Sigma_t + \Sigma_t \frac\alpha{2\alpha'}I\bigr)\Bigr]\\
            &= 2\Bigl[\bigl(\alpha\Sigma^{-1}+\frac\alpha{\alpha'}\Sigma_t^{-1}\bigr)^{-1} - c\Sigma_t\Bigr]\\
            &= 2\Bigl[c\bigl(\Sigma_t -\alpha'\Sigma_t\,(\Sigma+\alpha'\Sigma_t)^{-1}(\Sigma_t - \Sigma)\bigr) - c\Sigma_t\Bigr]\\
            &= -2c\alpha'\Sigma_t\,(\Sigma+\alpha'\Sigma_t)^{-1}(\Sigma_t-\Sigma),
        \end{align*}
        where we noted that $\gamma=2\frac{\alpha'}{\alpha'+1}$. This is~\cref{eq:lmm:compare:cbs:equiv:dot-Sigma'} accelerated with a factor $c$; hence, $\Sigma_t = \Sigma'_{ct}$.

        Again from~\cite{carrilloConsensusbasedSampling2022}, we obtain
        \begin{equation} \label{eq:lmm:compare:cbs:equiv:dot-m'}
            \frac\dd{\dd t} m'_t = -\alpha'\Sigma'_t\,(\Sigma + \alpha'\Sigma'_t)^{-1}(m'_t - m).
        \end{equation}
        For localized CBS,~\cref{eq:lmm:lcbs:gauss:moments:evol-m} becomes
        \begin{align*}
            \frac\dd{\dd t} m_t &= -2\alpha'c Q_t\Sigma^{-1}\,(m_t - m) = -c\alpha'\Sigma_t\,(\Sigma+\alpha'\Sigma_t)^{-1}(m_t-m)\\
            &= -c\alpha'\Sigma'_{ct}\,(\Sigma+\alpha'\Sigma'_{ct})^{-1}(m_t-m).
        \end{align*}
        This is the ODE~\cref{eq:lmm:compare:cbs:equiv:dot-m'} accelerated with a factor $c$. We conclude that $m_t = m'_{ct}$.
    \end{proof}
    For non-Gaussian distributions, this configuration of localized CBS likely performs quite poorly due to the large $\kappa$ value, which would result in excessive smoothing of the potential function $V$.

    As it concerns only Gaussian distributions, \cref{lmm:compare:cbs:equiv} is mainly of theoretical interest. A more general link between localized CBS and CBS is the $\kappa\to\infty$ limit from \cref{sec:cbs-approx}.

    \subsection{Polarized consensus-based sampling} \label{sec:compare:pcbs}
    A recently proposed modification of CBS, called polarized CBS~\cite{bungertPolarizedConsensusbasedDynamics2024}, makes the weight function dependent on the position~$U_t$ through a \emph{polarization kernel} $k(\cdot, \cdot)$. This results in the dynamics
    \begin{equation} \label{eq:compare:pcbs:pcbs}
        \dd U_t = -(U_t - \tilde\mu_{\alpha,k}(\rho_t; U_t))\,\dd t+\sqrt{2(\alpha+1)\,\tilde{\mathcal C}_{\alpha,k}(\rho_t; U_t)}\,\dd W_t,
    \end{equation}
    where we use the weights $\tilde w_{\alpha,k}(v; u) = k(v, u)\,\widehat\pi(v)^\alpha$ to define a polarized mean and covariance%
    \begin{subequations}
    \begin{align}
        \tilde\mu_{\alpha,k}(\rho; u) &\coloneqq \frac{\int v\tilde w_{\alpha,k}(v; u)\,\mathrm d\rho(v)}{\int \tilde w_{\alpha,k}(v; u)\,\mathrm d\rho(v)},\\
        \tilde{\mathcal C}_{\alpha,k}(\rho; u) &\coloneqq \frac{\int (v - \tilde\mu_{\alpha,k}(\rho; u))\otimes(v - \tilde\mu_{\alpha,k}(\rho; u))\,\tilde w_{\alpha,k}(v;u)\,\mathrm d\rho(v)}{\int \tilde w_{\alpha,k}(v; u)\,\mathrm d\rho(v)}. \label{eq:compare:pcbs:stat-pol:C}
    \end{align}
    \end{subequations}
    The main kernel used in~\cite{bungertPolarizedConsensusbasedDynamics2024} is Gaussian: for positive definite matrix $D$ and scalar $\lambda$,\hspace{-.1cm}
    \begin{equation} \label{eq:compare:pcbs:kernel-gauss}
        k(v, u) = \exp\Bigl(-(2\lambda)^{-1}\norm{v - u}_D^2\Bigr).
    \end{equation}
    This choice of $k$ introduces a localization effect: the polarized mean $\tilde\mu_{\alpha, k}(\rho_t; U_t)$---the ``target'' of the drift in~\cref{eq:compare:pcbs:pcbs}---tends to produce a value that is close to $U$, as well as having a large density in the posterior $\pi\propto\widehat\pi$. When $\pi$ is Gaussian and~\cref{eq:compare:pcbs:kernel-gauss} is used for the kernel, \mbox{\cite[Proposition~1]{bungertPolarizedConsensusbasedDynamics2024}} proves that $\pi$ is a stationary distribution of~\cref{eq:compare:pcbs:pcbs} with~\cref{eq:compare:pcbs:kernel-gauss}. More generally, a localized kernel attempts to lessen the Gaussian assumption present in CBS in order to allow multimodal sampling.

    \Cref{eq:compare:pcbs:pcbs} is very reminiscent of the localized CBS dynamics~\cref{eq:lcbs:loc-prec:lcbs}. We now compare the two, highlighting three main differences. \begin{enumerate}[label={(\roman*)}]
        \item Localized CBS includes the correction term $\nabla_u\cdot\mathbf P_{U,t}$. As $\beta\rightarrow\infty$ and~${\kappa\rightarrow0}$, we expect convergence to the interacting Langevin dynamics, under which even non-Gaussian and multimodal $\pi$ are invariant. In polarized CBS, we do not expect a similar convergence.
        \item By including the preconditioner $\mathbf P_{U,t}$ in the distance weighting for the localized mean, localized CBS has affine-invariant dynamics when $\mathbf P_{U,t}$ is a weighted or unweighted covariance. The kernels proposed for polarized CBS in~\cite{bungertPolarizedConsensusbasedDynamics2024} are not affine-invariant as they weight the distance between two particles by a fixed matrix.
        
        One could modify the kernel in polarized CBS to make it affine-invariant, e.g., by replacing $D$ by $\mathcal C(\rho_t)$. While interesting, this would not change the other benefits of localized CBS.

        \item Using $\widehat\pi$-based and distance-based weights in the (relatively cheap) mean computation is crucial in both methods, since that drives evolution towards high-probability regions and allows multimodal sampling. On the other hand, we find that the benefits of using weights in the (relatively expensive) covariance preconditioner are significantly less clear. Yet, this can be costly: distance-weighted statistics need to be computed for each of the $J$ particles, instead of once, globally. Polarized CBS uses the same weighting for both the mean and the covariance, while localized CBS decouples these two elements and allows to use a cheap, unweighted covariance as a preconditioner. This can make localized CBS the~\mbox{computationally}~cheaper~method in~higher-dimensional~problems.
    \end{enumerate}

    \subsection{Localized ALDI} \label{sec:compare:laldi}
    The gradient-free McKean--Vlasov dynamics underlying the closely related EKS and ALDI algorithms~\cite{garbuno-inigoInteractingLangevinDiffusions2020,garbuno-inigoAffineInvariantInteracting2020b} are described by
    \begin{equation} \label{eq:compare:laldi:aldi}
        \dd U_t = -\mathcal C^G(\rho_t)\Gamma^{-1}\,(G(U_t) - y)\,\dd t + \mathcal C(\rho_t)\Gamma_0^{-1}\,(U_t - u_0)\,\dd t + \sqrt{2\mathcal C(\rho_t)}\,\dd W_t,
    \end{equation}
    where
    \begin{equation}
        \mathcal\mu^G(\rho) \coloneqq \medint\int G(v)\,\mathrm d\rho(v) \quad \text{and} \quad \mathcal C^G(\rho) \coloneqq \medint\int (v - \mu(\rho))\otimes(G(v) - \mu^G(\rho))\,\mathrm d\rho(v).
    \end{equation}
    \Cref{eq:compare:laldi:aldi} approaches the sampling problem~\cref{eq:intro:intro:dist} explicitly for a Bayesian inverse problem~\cref{eq:intro:bip:bip}, estimating the parameters of a model $G$. It assumes that ${\pi_\mathrm{prior} = \mathcal N(u_0, \Gamma_0)}$ and~${\pi_\mathrm{noise} = \mathcal N(0, \Gamma)}$ are Gaussians. When $G$ is linear,~\cref{eq:compare:laldi:aldi} is equivalent to~\cref{eq:ips:lang:lang-int} and, hence, has $\pi$ as an invariant distribution~\cite{garbuno-inigoInteractingLangevinDiffusions2020}. For nonlinear $G$, this is generally not the case.

    In~\cite{reichFokkerPlanckParticleSystems2021a}, localization with a scalar $\lambda$ and positive definite matrix $D$ is added to~\cref{eq:compare:laldi:aldi}, together with a correction term similar to the one in~\cref{eq:ips:lang:lang-C}. This can be seen as Langevin diffusions where the gradients are approximated using the method in~\cite{wackerPerspectivesLocallyWeighted2024a}. The resulting dynamics are
    \begin{equation}
    \begin{aligned}
        \dd U_t &= -\tilde{\mathcal C}_{\lambda D}^G(\rho_t; U_t)\Gamma^{-1}\,(G(U_t) - y)\,\dd t - \tilde{\mathcal C}_{\lambda D}(\rho_t; U_t)\Gamma_0^{-1}\,(U_t - u_0)\,\dd t\\
        &\qquad + \nabla_u \cdot \tilde{\mathcal C}_{\lambda D}(\rho_t; U_t) + \sqrt{2\tilde{\mathcal C}_{\lambda D}(\rho_t; U_t)}\,\dd W_t,
    \end{aligned}
    \end{equation}
    where we use $\tilde w_{\lambda D}(v; u) = \exp(-(2\lambda)^{-1}\norm{v - u}_D^2)$ to define
    \begin{subequations}
    \begin{align}
        \tilde\mu_{\lambda D}(\rho; u) &\coloneqq \frac{\int v\tilde w_{\lambda D}(v; u)\,\mathrm d\rho(v)}{\int \tilde w_{\lambda D}(v; u)\,\mathrm d\rho(v)},\\
        \tilde{\mathcal C}_{\lambda D}(\rho; u) &\coloneqq \frac{\int (v - \tilde\mu_{\lambda D}(\rho; u))\otimes(v - \tilde\mu_{\lambda D}(\rho; u))\,\tilde w_{\lambda D}(v;u)\,\mathrm d\rho(v)}{\int \tilde w_{\lambda D}(v; u)\,\mathrm d\rho(v)}.
    \end{align}
    \end{subequations}
    The definition of $\tilde{\mathcal C}_{\lambda D}^G$ is analogous. Localization in ALDI (which, we stress again, differs from the concept of localization discussed in \cref{rem:intro:lit:loc}) improves performance for nonlinear $G$. However, we point out two potential comparative advantages of the localized CBS dynamics. \begin{enumerate}[label={(\roman*)}]
        \item Similarly to polarized CBS, when the matrix $D$ in localized ALDI is fixed, the dynamics are not affine-invariant. In~\cite{reichFokkerPlanckParticleSystems2021a}, affine-invariance is proven if $D = \mathcal C(\rho_0)$, such that $D$ is also subject to any affine transformation. This can be considered a slightly weaker form of affine-invariance: in practice, the scaling of the parameters is often not fully known and $\rho_0$ cannot be chosen with optimal scaling. After all, if the scaling was known exactly, we could rescale the problem manually and would not need affine-invariant samplers. Whereas the effect of the initial distribution may diminish over time in localized CBS (see \cref{sec:num:multi} and \cref{fig:num:multi:ai} in particular), choosing ${D=\mathcal C(\rho_0)}$ in localized ALDI ensures that an incorrectly scaled~initial~distribution~persists.
        
        Analogously to \cref{sec:compare:pcbs} for polarized CBS, it is conceivable that replacing $D$ by $\mathcal C(\rho_t)$ could improve the localized ALDI algorithm by making it affine-invariant.

        \item As in polarized CBS, localized ALDI uses weighted covariance matrices, which can be more expensive to compute than unweighted ones in localized CBS.
    \end{enumerate}

    \subsection{Other methods} \label{sec:compare:other}
    The localized CBS dynamics~\cref{eq:lcbs:loc-prec:lcbs} arise from a specific approximation of the gradient in preconditioned Langevin dynamics, as detailed in \cref{sec:cbs-approx}. Several other gradient-free Langevin-type variants have been proposed in the literature. Besides localized ALDI, discussed in \cref{sec:compare:laldi}, notable examples include the following methods. \begin{enumerate}[label={(\roman*)}]
        \item The method of~\cite{dingConstrainedEnsembleLangevin2022} approximates gradients using only particles within a fixed neighborhood of the target particle. The authors of~\cite{dingConstrainedEnsembleLangevin2022} report that this approach can become unstable because isolated particles may experience unrealistically large approximate gradients. To mitigate this, they propose evaluating the true gradient at such particles while the remaining ones use the ensemble-based approximation. Although conceptually interesting—and potentially adaptable to localized CBS—this algorithm is not fully gradient-free and is therefore applicable in a more limited setting than localized CBS.
        \item The multiscale sampler of~\cite{pavliotisDerivativeFreeBayesianInversion2022} introduces two types of particles: a main (\emph{macro}) particle that represent the distribution and a local swarm of auxiliary (\emph{micro}) particles used to estimate gradients. This construction can handle non-Gaussian targets and may scale favorably with problem dimension because the particles stay tightly clustered. However, only the macro particle contributes to sampling; the computational effort invested in the auxiliary particles serves solely for gradient estimation. Moreover, since all particles in the swarm remain close, the method can explore only one mode at a time, which leads to metastability in multimodal problems. In contrast, every particle in localized CBS contributes to the sample ensemble, and the method can explore multiple modes simultaneously.
        \item There exist many Langevin-type samplers using the proximal operator (e.g.,~\cite{doi:10.1137/16M1108340,leeStructuredLogconcaveSampling2021,pereyraProximalMarkovChain2016,titsiasAuxiliaryGradientBasedSampling2018}), often with the goal of sampling from non-differentiable distributions. These are not ensemble methods and they are neither affine-invariant nor gradient-free. However, they typically scale very well to high dimensions, which is useful in applications such as Bayesian imaging.
        
        Related to these algorithms,~\cite{doi:10.1137/21M1406349} uses \emph{Tweedie's identity}
        \begin{equation}
            \nabla \widetilde V^\kappa(u) = \kappa^{-1}(u - \widetilde D^\kappa(u))
        \end{equation}
        with
        \begin{equation}
            \widetilde V^\kappa(u) = -\log\left(\int \exp(-\norm{u-v}_2^2/(2\kappa))\exp(-V(v))\,\mathrm dv\right)
        \end{equation}
        and the \emph{denoiser}
        \begin{equation} \label{eq:compare:other:tweedie-D}
            \widetilde D^\kappa(u) = \frac{\int v \exp(-\norm{u-v}_2^2/(2\kappa))\exp(-V(v))\,\mathrm dv}{\int \exp(-\norm{u-v}_2^2/(2\kappa))\exp(-V(v))\,\mathrm dv}
        \end{equation}
        to replace the gradient of a nondifferentiable part of the target potential by $\kappa^{-1}(u - \widetilde D^\kappa(u))$. Tweedie's identity offers a second two-step approximation scheme to retrieve~\cref{eq:cbs-approx:approx-2} from interacting Langevin: approximate $V$ by $\widetilde V^\kappa$ (with certain adaptations for weighting the exponential by $\beta$ and the distance by $\mathbf P_t$) to apply Tweedie's identity, and then approximate the integrals over the Lebesgue measure in~\cref{eq:compare:other:tweedie-D} by integrals over $\rho_t$.
    \end{enumerate}

\section{Discretizations and computational considerations} \label{sec:comput}
    We now bridge the gap between the theoretical localized CBS dynamics~\cref{eq:lcbs:loc-prec:lcbs} and practical algorithms to sample from~\cref{eq:intro:intro:dist}.

    \paragraph{Particle discretization.} The dynamics~\cref{eq:lcbs:loc-prec:lcbs} of localized CBS are defined as a McKean--Vlasov SDE: the law of $U_t$ factors into its evolution. Of course, this law is not available in practice and~\cref{eq:lcbs:loc-prec:lcbs} does not yet constitute an algorithm. We now introduce particle approximations, i.e., schemes where $J$ approximate instances of~\cref{eq:lcbs:loc-prec:lcbs} are run in parallel. They use their own empirical distribution instead of their unknown exact law to drive their evolution. We denote the $i$th particle by $U_t^i$ and the entire ensemble by the $d\times J$ matrix $\bm U_t = (U_t^1, U_t^2, \ldots, U_t^J)$.

    When $\mathbf P(\rho_t; U_t)$ is an unweighted or weighted covariance, we replace it by a particle approximation: a (weighted) sample covariance of all particles defined as
    \begin{equation}
        \widehat{\mathcal C}(\bm U) \coloneqq \frac1J\sum\nolimits_{j=1}^J (U^j - \widehat\mu(\bm U))\otimes(U^j - \widehat\mu(\bm U)) \qquad \text{with} \qquad \widehat\mu(\bm U) \coloneqq \frac1J\sum\nolimits_{j=1}^J U^j
    \end{equation}
    for the unweighted covariance and, for the weighted covariance,
    \begin{equation}
        \widehat{\mathcal C}_{\alpha, \lambda\widehat{\mathcal C}(\bm U)}(\bm U; U^i) \coloneqq\sum\nolimits_{j=1}^J \omega_{\alpha,\lambda\widehat{\mathcal C}(\bm U)}^{ij}\;(U^j - \widehat\mu_{\alpha, \lambda\widehat{\mathcal C}(\bm U)}(\bm U; U^i))\otimes(U^j - \widehat\mu_{\alpha, \lambda\widehat{\mathcal C}(\bm U)}(\bm U; U^i))
    \end{equation}
    with
    \begin{equation} \label{eq:comput:particle:mu}
        \widehat\mu_{\alpha, A}(\bm U; U^i) \coloneqq \sum\nolimits_{j=1}^J\omega_{\alpha, A}^{ij} U^j,
    \end{equation}
    where $\omega_{\alpha, A}^{ij} = w_{\alpha, A}^{ij} / \sum_{k = 1}^J w_{\alpha, A}^{ik}$ with $w_{\alpha, A}^{ij} = w_{\alpha, A}(U_j; U_i)$. In general, we denote the particle approximation to $\mathbf P(\rho; U^i)$ by $\widehat{\mathbf P}(\bm U; U^i)$. As in the mean-field setting of \cref{sec:lcbs}, we assume that this preconditioner is invertible. Notably, this requires that the $J$ particles span the entire space $\bR^{d}$. Particle-based methods such as ALDI have been used with $J \ll d$ to sample in a linear subspace~\cite{garbuno-inigoAffineInvariantInteracting2020b}; we leave the application of localized CBS to those settings to~future~work.

    We also replace the weighted mean in~\cref{eq:lcbs:loc-prec:lcbs} by a particle approximation. The most straightforward option is a weighted sample mean of all particles, given in~\cref{eq:comput:particle:mu}. Later in this section, we will also explore a random-batch modification of this scheme.

    \begin{remark}[Scaling of the particle discretization to higher dimensions] \label{rem:comput:scale}
        The performance of localized CBS in high dimensions has not been theoretically analyzed. Since the method relies on local interactions between particles distributed throughout parameter space for gradient estimations, increasing dimensionality may increase the typical distance between particles and hence negatively affect performance. Clarifying this requires dedicated study.
    \end{remark}

    \paragraph{Time discretization.} After particle discretization, we are left with a system of interacting SDEs. Since time discretization is not a main point of focus of this paper, we use the standard Euler--Maruyama discretization with $N$ timesteps of length $\dt$ in our numerical tests, although others are possible\footnotemark{}. We denote the $i$th particle at timestep $n$ by $U_n^i$ and the ensemble by $\bm U_n$.

    \footnotetext{Classical CBS uses a non-standard timestepping scheme, which helps them prove properties in the time-discrete setting. Polarized CBS adopts this same method. While EKS uses a linearly implicit split-step time discretization, both ALDI and localized ALDI use standard Euler--Maruyama.}

    \paragraph{Random-batch interaction.} For the weighted mean in the drift term of localized CBS, we introduce a particle approximation that forbids self-to-self interaction (which would have large weights when particles are sparse) and randomly subsamples particles with a factor $\nu\in(0, 1]$. Consider
    \begin{equation} \label{eq:particle:mu-batch}
        \widehat\mu_{\beta,\kappa\widehat{\mathbf P}}^{(\nu,i)}(\bm U_n; U_n^i) \coloneqq \frac{\sum_{j=1}^J \iota(i,j, n) \,U_n^j\,w_{\alpha,\lambda\widehat{\mathbf P}}^{ij}}{\sum_{j=1}^J \iota(i,j, n)\,w_{\alpha,\lambda\widehat{\mathbf P}}^{ij}},
    \end{equation}
    where $\iota$ selects whether to count the contribution of particle $j$:
    \begin{equation}
        \iota(i,j, n) \coloneqq \begin{cases}
            0, & \text{if $i=j$ or $\theta_n^{i,j} > \nu$, with $\theta_n^{i,j}\sim \mathcal{U}(0,1)$ i.i.d.},\\
            1, & \text{otherwise}.
        \end{cases}
    \end{equation}
    Here we write $\widehat{\mathbf P} \coloneqq\widehat{\mathbf P}(\bm U_n; U_n^i)$ for notational convenience. This idea is related to the random-batch method of~\cite{jinRandomBatchMethods2020a}, where interacting-particle systems are divided into small, random batches for a reduction in computational cost. We will choose relatively large values of $\nu$ (e.g., $\nu=0.5$), so the cost difference is moderate; instead, we are motivated by our observation that random-batch interaction can improve sampling performance for higher-dimensional, multimodal problems. This is illustrated in \cref{sec:num}.

    \paragraph{Square-root-free formulation.} Similarly to what is proposed in~\cite{garbuno-inigoAffineInvariantInteracting2020b}, it is possible to replace the symmetric positive definite square root $\sqrt{\widehat{\mathbf P}}\in\bR^{d\times d}$ by a (rectangular) Cholesky factor~${\widehat{\mathbf P}^{1/2}\in\bR^{d\times n}}$, for any $n$, with the property that
    \begin{equation}
        \widehat{\mathbf P} = \widehat{\mathbf P}^{1/2}\,(\widehat{\mathbf P}^{1/2})^T.
    \end{equation}
    By using $\widehat{\mathbf P}^{1/2}$ and simultaneously replacing the $d$-dimensional Brownian motion $W_t$ by an $n$-dimensional one, the Fokker--Planck equation remains the same. Henceforth, let us define~${\bar U\coloneqq \widehat\mu(\bm U)}$,~$\bar U^{(i)}\coloneqq \widehat\mu_{\alpha, \lambda\widehat{\mathcal C}(\bm U)}(\bm U; U^i)$, $w^{ij}\coloneqq w^{ij}_{\alpha,\lambda\widehat{\mathcal C}(\bm U)}$, and $\omega^{ij}\coloneqq \omega^{ij}_{\alpha,\lambda\widehat{\mathcal C}(\bm U)}$ for conciseness. The unweighted and weighted covariance have factors
    \begin{align}
        \widehat{\mathcal C}^{1/2}(\bm U) &= \frac1{\sqrt J}\begin{bmatrix}
            (U^1 - \bar U) & \ldots & (U^J - \bar U)
        \end{bmatrix}\\
        \widehat{\mathcal C}_{\alpha,\lambda\widehat{\mathcal C}(\bm U)}^{1/2}(\bm U; U^i) &= \begin{bmatrix}
            \sqrt{\omega^{i1}}\,(U^1 - \bar U^{(i)}) & \ldots & \sqrt{\omega^{iJ}}\,(U^J - \bar U^{(i)})
        \end{bmatrix},
    \end{align}
    respectively. These quantities are readily available and avoid computing the square root of $\widehat{\mathbf P}$. This should offer computational advantages unless $J\gg d$.

    \paragraph{Covariance-weighted norms.} Within the weight function~\cref{eq:cbs-approx:w-and-mu-w}, localized CBS uses squared norms~${\norm{U^i - U^j}_{\widehat{\mathbf P}}^2= (U^i - U^j)^T\widehat{\mathbf P}^{-1}(U^i - U^j)}$ weighted by the preconditioner. This quantity can be evaluated as follows: precompute
    \begin{equation*}
        M \coloneqq \bm U^T\widehat{\mathbf P}^{-1}\bm U
    \end{equation*}
    and use
    \begin{equation*}
        \norm[\big]{U^i - U^j}_{\widehat{\mathbf P}}^2 = M_{ii} + M_{jj} - 2M_{ij}.
    \end{equation*}
    If $\widehat{\mathbf P}$ depends on $i$, then so does $M$.

    \paragraph{Correction term.} The discretized dynamics contain the correction term~${\nabla_{u^i}\cdot\widehat{\mathbf P}}$, where we write~${\widehat{\mathbf P} \coloneqq \widehat{\mathbf P}(\bm U_n; U_n^i)}$. We now study this term for the preconditioners we consider in this article. When $\widehat{\mathbf P} = K$, clearly $\nabla_{u^i}\cdot\widehat{\mathbf P} = 0$; when $\widehat{\mathbf P} = \widehat{\mathcal C}_{\alpha, \lambda\widehat{\mathcal C}(\bm U)}(\bm U; U)$, the correction term is given in \cref{lmm:comput:corr}. The preconditioners $\widehat{\mathcal C}_\alpha(\bm U)$ and $\widehat{\mathcal C}(\bm U)$ correspond to \cref{lmm:comput:corr} with~${\lambda = \infty}$ and, in the latter case, $\alpha=0$ as well. Note that the correction term contains a derivative of $\widehat\pi$ when $\alpha\ne0$, which is problematic in our otherwise derivative-free method. In this paper we always choose $\alpha=0$. If situations are identified where $\alpha\ne0$ is preferred, one could neglect the $\widehat\pi$ derivatives (see \cref{rem:comput:ignore-corr-terms}) or construct a particle approximation (e.g., based on the proximal operator such as in \cref{sec:cbs-approx} or on the technique in~\cite{schillingsEnsembleBasedGradientInference2023a}).~This~is~left~for~future~work.

    \begin{lemma} \label{lmm:comput:corr}
        When $\widehat{\mathbf P} = \widehat{\mathcal C}_{\alpha, \lambda\widehat{\mathcal C}(\bm U)}(\bm U; U^i)$, it holds that
        \begin{equation} \label{eq:lmm:comput:corr:expr}
        \begin{aligned}
            &\nabla_{u^i}\cdot\widehat{\mathbf P} = \omega^{ii}\,(d+1)U_c^{(i),i} + \frac1\lambda \bm U_c^{(i)} \left(\omega^i \odot \mathrm{diag}\bigl((\bm U_c^{(i)})^T \widehat{\mathcal C}^{-1} \bm U_c^{(i)}\bigr)\right)\\
            & \quad - \frac1{\lambda J}\widehat{\mathbf P} \widehat{\mathcal C}^{-1}\left(U_c^{(i),i}(U_c^{(i),i})^T+\widehat{\mathbf P}\right)\widehat{\mathcal C}^{-1} U_c^i\\
            & \quad + \frac1{\lambda J}\bm U_c^{(i)}\,\bigl(\omega^i \odot \mathrm{diag}\bigl((\bm U_c^{(i)})^T \widehat{\mathcal C}^{-1} (\bm U - U^i\bm 1^T)\bigr) \odot \bigl((U_c^i)^T \widehat{\mathcal C}^{-1} (\bm U - U^i\bm 1^T)\bigr)^T\bigr)\hspace{-.3cm}\\
            & \quad + \alpha\omega^{ii}\left(U_c^{(i),i}(U_c^{(i),i})^T - \widehat{\mathbf P}\right)\frac{\nabla_{u^i}\widehat\pi(U^i)}{\widehat\pi(U^i)},
        \end{aligned}
        \end{equation}
        where we denote by $\odot$ the Hadamard (elementwise) product. Furthermore, $\bm U_c \coloneqq \bm U - \bar U\bm 1^T$ and ${\bm U_c^{(i)} \coloneqq \bm (\bm U - \bar U^{(i)}\bm 1^T)}$. We write $U_c^i$ for the $i$th column of~$\bm U_c$, $U_c^{(i),j}$ for the $j$th column of~$\bm U_c^{(i)}$, $\widehat{\mathcal C}$ as a shorthand for $\widehat{\mathcal C}(\bm U)$, and $\omega^i = (\omega^{i1}, \ldots, \omega^{iJ})^T$.
    \end{lemma}
    \begin{proof}
        This is proven in \cref{sec:apdx}.
    \end{proof}
    \begin{remark} \label{rem:comput:ignore-corr-terms}
        All terms in~\cref{eq:lmm:comput:corr:expr} except the second go down in magnitude as the number of particles $J$ increases (through either $1/J$ or $\omega^{ii}=w^{ii}/\sum_jw^{ij}$, which depends on $J$ due to the normalization). It may be computationally interesting to omit some of these terms in practice when $J$ is large.
    \end{remark}

    \paragraph{Localized CBS algorithm.} Combining the points discussed in this section leads to the localized CBS algorithm
    \begin{equation} \label{eq:comput:lcbs-alg}
        U_{n+1}^i = U_n^i + \Bigl[-\frac\gamma\kappa\Bigl(U_n^i - \widehat\mu_{\beta,\kappa\widehat{\mathbf P}}^{(\nu,i)}(\bm U_n; U_n^i)\Bigr) + \nabla_{u^i}\cdot\widehat{\mathbf P}\Bigr]\dt + \sqrt{2\dt}\,\widehat{\mathbf P}^{1/2}\xi_n^i,
    \end{equation}
    with $\xi_n^i\sim\mathcal N(0, I)$ i.i.d.\ random variables. Here, $(\beta, \kappa)$ are free parameters and $\gamma$ is chosen based on them (e.g., through \cref{cor:lcbs:gauss:gamma}). The choice of preconditioner $\widehat{\mathbf P} \coloneqq \widehat{\mathbf P}(\bm U_n; U_n^i)$ is free; examples are a constant matrix, an unweighted covariance as in ALDI, a weighted covariance as in CBS, or a localized weighted covariance. The correction term $\nabla_{u^i}\cdot\widehat{\mathbf P}$ is detailed in \cref{lmm:comput:corr,rem:comput:ignore-corr-terms}; for many choices of $\widehat{\mathbf P}$, it is negligible or zero.

    \begin{remark}[Parallelizability]
        Earlier, we motivated particle methods such as localized CBS by their potential for parallelization. The algorithm~\cref{eq:comput:lcbs-alg} is indeed parallelizable across particles: $\{V(U_n^i)\}_{i=1}^J$ can be computed independently. To compute the interaction terms $\widehat\mu_{\beta,\kappa\widehat{\mathbf P}}^{(\nu,i)}$ and (potentially) $\widehat{\mathbf P}$, communication between particles is required. However, in practical applications, $V$ is typically an expensive model (e.g., a PDE solver) whose evaluations are the main bottleneck.
    \end{remark}

\section{Numerical results} \label{sec:num}
    This section applies the localized CBS algorithm to a variety of problems, illustrating its properties. We compare our method to several existing alternatives: CBS (to illustrate that it is not suited for non-Gaussian problems) on the one hand and polarized CBS and localized ALDI on the other (to compare our algorithm to the two most similar options). Other algorithms are often challenging to compare to in a fair way: they either use gradients, are less parallelizable, or produce fewer samples per evaluation of $V$.

    For both localized CBS and alternative methods, we will use the Euler--Maruyama time discretizations discussed in \cref{sec:comput}. Unless mentioned otherwise, localized CBS uses timestep $\dt=0.01$ and preconditioner $\mathcal C(\rho_t)$, an unweighted covariance. Except for \cref{fig:num:gaussian:fig}, we always use the $\gamma$ value implied by the preconditioner through \cref{cor:lcbs:gauss:gamma}. The default random-batch parameter is $\nu=1$, such that a particle interacts with all other particles. In addition, except in \cref{sec:num:tent}, we sample by running the specified algorithm $16$ times and combining the particle positions at all time steps within the final quarter of the time interval in each run. Samples are plotted with the \verb+density+ function from the Julia package \verb+StatsPlots+, which performs kernel density estimation (KDE) to smooth the samples into an approximate distribution. When the dimension~$d$ is larger than one, we plot the marginal distribution in the first dimension unless specified otherwise. The initial distribution for all sampling methods is a zero-centered Gaussian $\mathcal N(0, \Sigma_0)$; by default, ${\Sigma_0 = I/2}$.~Our~code~is~\mbox{located}~at 
    \begin{quote}
        \texttt{\url{https://gitlab.kuleuven.be/numa/public/paper-code-lcbs}}.
    \end{quote}

    \subsection{Gaussian distributions} \label{sec:num:gaussian}
    We apply localized CBS to a Gaussian distribution, to verify the values of $\gamma$ we obtained in \cref{sec:lcbs:gauss} and to assess the influence of an incorrect $\gamma$ on the sampled distribution. We consider the one-dimensional~potential~function\hspace{-.1cm}
    \begin{equation}
        V(u) = u^2
    \end{equation}
    and use $(\beta, \kappa) = (2.0, 0.01)$, $J=500$ particles, and $N=200$~timesteps. The top row of \cref{fig:num:gaussian:fig} shows\footnote{In the figures of this section, we abbreviate localized CBS, polarized CBS, and localized ALDI as \emph{LCBS}, \emph{PCBS}, and \emph{LALDI}, respectively.} the resulting samples with $\gamma\in\{0.5, \bar\gamma, 1.0, 1.5\}$, where ${\bar\gamma\approx0.677}$ is the value given by~\cref{eq:cor:lcbs:gauss:gamma:gamma-C}---since we use our default preconditioner $\mathcal C(\rho_t)$, this should ensure correct sampling for Gaussians. When $\gamma<\bar\gamma$, there is insufficient drift towards the weighted mean and the samples are too spread out. When $\gamma>\bar\gamma$, the opposite effect occurs. As predicted by the theory, the distribution is sampled correctly~with~$\gamma=\bar\gamma$.

    \begin{figure}
        \centering
        \includegraphics[width=.78\textwidth]{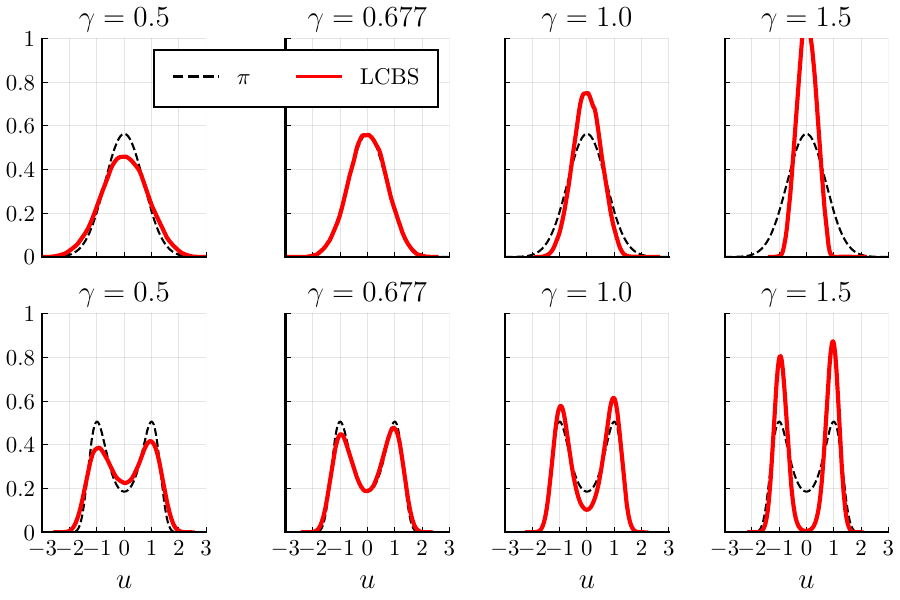}
        \caption{Localized CBS samples for a one-dimensional Gaussian target distribution (top row) and a one-dimensional multimodal target distribution (bottom row)}
        \label{fig:num:gaussian:fig}
    \end{figure}

    \subsection{Multimodal distributions} \label{sec:num:multi}
    One of the design objectives of localized CBS is to handle non-Gaussian, potentially even multimodal, distributions. The first multimodal distribution that we consider corresponds to the potential function
    \begin{equation}
        V(u) = \norm{u^2 - 1}_2^2,
    \end{equation}
    where $u^2$ applies the exponent elementwise. The resulting distribution $\pi$ is bimodal when the dimension $d$ is 1 and a tensor product of bimodal distributions otherwise.

    We first apply localized CBS with the same parameters from \cref{sec:num:gaussian} to this potential in dimension $d=1$, with the results shown in the bottom row of \cref{fig:num:gaussian:fig}. Even though $\pi$ is not sampled exactly, which is expected due to the small $\beta$ value, using the $\gamma=\bar\gamma$ that is correct for Gaussian distributions seems to be a good heuristic even in this non-Gaussian case.

    For this potential with ${d\in\{1,10\}}$, we compare localized CBS (with our default parameters, along with ${\beta=10}$ and~${\kappa=0.01}$ when $d=1$ or $\kappa=0.03$ when $d=10$) to CBS~(with~${\alpha=10}$ and~${\dt=0.01}$), polarized CBS (with~${\alpha=10}$,~${\lambda = 0.005}$ when~$d=1$ or~$\lambda=0.1$ when~$d=10$,~$D=I$, and~${\dt = 0.01}$), and localized ALDI (with $\lambda=0.02$ when~$d=1$ or $\lambda=0.4$ when $d=10$, $D=I$, and~$\dt=0.05$). These parameters were chosen to the best of our abilities to give good results for each method for the chosen values of $J=200$ particles and~${N=1000}$ timesteps. For localized CBS, we additionally compare $\nu=1$ to~${\nu=0.5}$. \Cref{fig:num:multi:comp} shows the resulting samples. Localized CBS accurately finds the distribution when~$d=1$; when~${d=10}$, we get a good approximation only with random-batch interaction turned on ($\nu=0.5$). Localized ALDI and polarized CBS each give relatively close approximations, but are unable to exactly fit $\pi$ when $d=10$. We also noted that for these two methods, the performance was significantly more sensitive to the choice of parameters than for localized CBS. Classical CBS finds the multiple modes only when~${d=10}$, and fails to approximate the density in between.\footnote{Note that the particles in CBS always follow a Gaussian distribution; in our experiment, CBS captures multiple modes as the results from $16$ runs (which may find different modes) are aggregated.}

    \begin{figure}
        \centering
        \includegraphics[width=.8\textwidth]{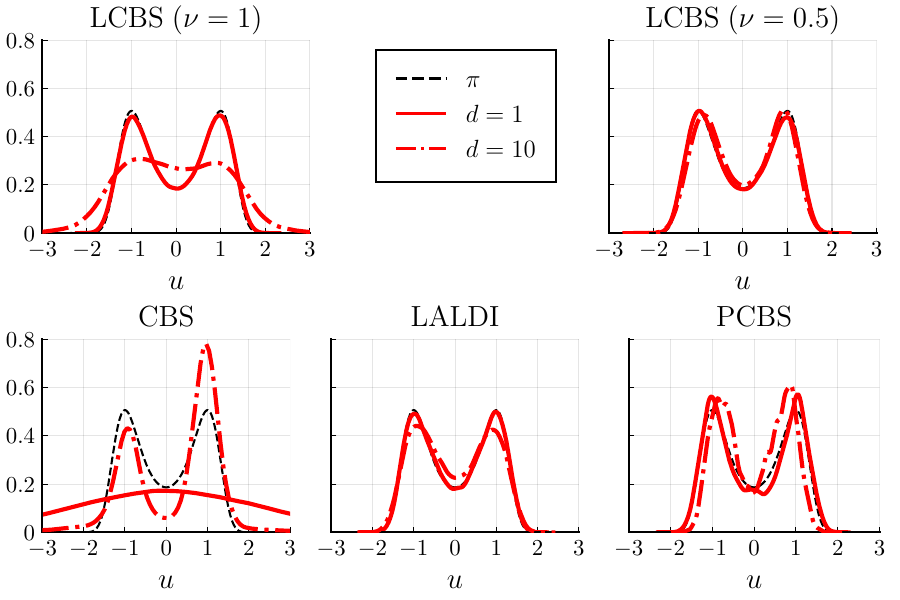}
        \caption{Marginalized samples for a 1- or 10-dimensional multimodal target distribution. Recall that we aggregate the samples from 16 runs, which explains the non-Gaussian CBS samples.}
        \label{fig:num:multi:comp}
    \end{figure}

    Secondly, we study affine-invariance of the three best performing algorithms: localized CBS, polarized CBS, and localized ALDI. To this end, consider
    \begin{equation} \label{eq:num:multi:V-ai}
        V(u) = \norm{(\sqrt\Lambda\,u)^2 - 1}_2^2,
    \end{equation}
    where $d=2$ and the square again applies elementwise. The matrix~${\Lambda = \mathrm{diag}(1, 10^4)}$ changes the scale of the second parameter. Suppose that we may not know this scaling before solving the problem. Without knowing the form of the potential~\cref{eq:num:multi:V-ai}, we assume that the relative scaling of the two parameters is given by a matrix $\tilde\Lambda$ that can either be the correct $\tilde\Lambda=\Lambda$ or the incorrect $\tilde\Lambda=I$. We use this assumption to choose an initial \mbox{covariance}~${\Sigma_0 = \tilde\Lambda^{-1}/2}$.~\mbox{In addition},
    \begin{enumerate}[label={(\roman*{})}]
        \item in polarized CBS and localized ALDI, we set $D=\tilde\Lambda^{-1}$;
        \item in our localized CBS algorithm, we simply use the preconditioner $\mathcal C(\rho_t)$, which does not need \emph{a priori} scaling information.
    \end{enumerate}
    We use $(\beta, \kappa) = (10, 0.03)$ for localized CBS, $\lambda = 0.02$ and $\dt = 0.05$ for localized ALDI, and~${\lambda = 0.005}$ for polarized CBS. \Cref{fig:num:multi:ai} shows the samples by these algorithms for both~$\tilde\Lambda$ cases. Polarized CBS and localized ALDI perform as expected when the scaling is guessed correctly, but generate clearly incorrect samples when it is not. This is because they depend on $\tilde\Lambda$ not only for their initial condition, but also for weighting the distance between particles throughout the algorithm. Localized CBS, on the other hand, proves itself able to recover from this poor initial guess and produces good---and almost identical---samples~for~both~$\tilde\Lambda$s.

    \begin{figure}
        \centering
        \includegraphics[width=.8\textwidth]{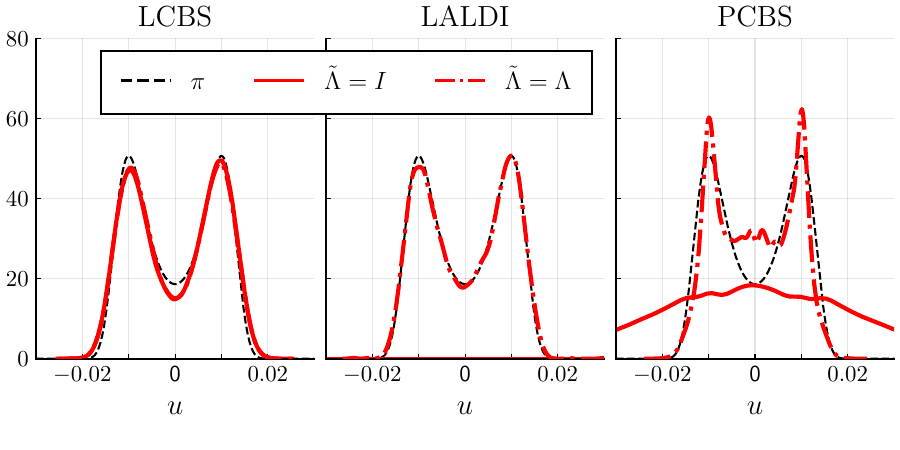}
        \caption{Marginalized samples for the poorly scaled distribution with potential~\cref{eq:num:multi:V-ai}. In the left figure, the two red lines are visually indistinguishable; in the middle figure, one line coincides with the horizontal axis.}
        \label{fig:num:multi:ai}
    \end{figure}

    \begin{figure}
        \centering
        \includegraphics[width=.8\textwidth]{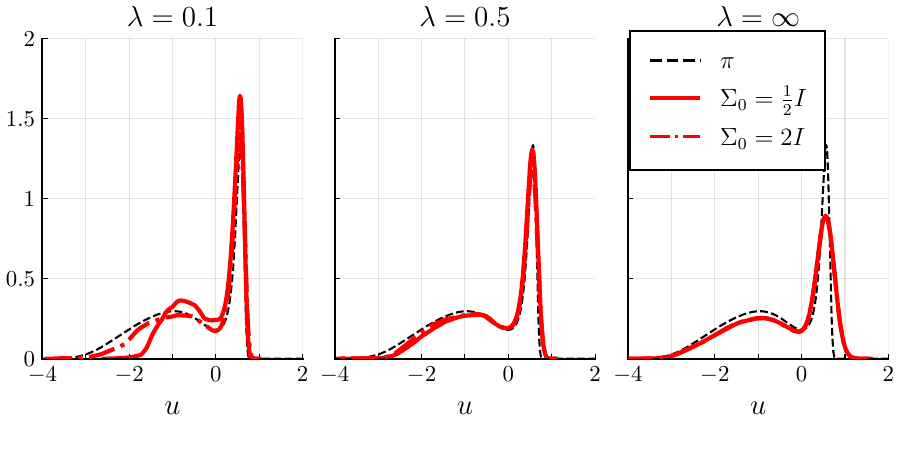}
        \caption{Localized CBS samples for~\cref{eq:num:multi:diffpeaks} with various preconditioners and initial distributions}
        \label{fig:num:multi:wcov}
    \end{figure}

    Now consider a different potential function with $d=1$:
    \begin{equation} \label{eq:num:multi:diffpeaks}
        V(u) = 2(u\exp(u))^4 - 4(u\exp(u))^2 - 2(u/3)^5 + 2.
    \end{equation}
    This again produces a multimodal distribution, now with one wide and one narrow peak. We first compare preconditioners for localized CBS with parameters ${(\beta, \kappa) = (10, 0.02)}$, ${N=1000}$ timesteps, and~${J=200}$ particles. We use preconditioner $\mathcal C_{\alpha, \alpha\lambda\mathcal C(\rho_t)}(\rho_t; U_t)$ for~${\lambda\in\{0.1, 0.5, \infty\}}$ and $\alpha\to0$; when $\lambda=\infty$, this equals the unweighted covariance $\mathcal C(\rho_t)$. The initial distribution has covariance ${\Sigma_0\in\{I/2, 2I\}}$. Resulting samples are shown in \cref{fig:num:multi:wcov}. When~$\lambda=\infty$, the global nature of the preconditioner means that the presence of the wide peak causes the narrow one to be estimated poorly. When~$\lambda=0.1$, on the other hand, localized CBS struggles to find the full distribution when the initial distribution is too narrow. A moderate $\lambda=0.5$ seems to combine the benefits of both extremes. Note that, as discussed in \cref{sec:lcbs:loc-prec},~$\lambda=\infty$ has the computational advantage that only one covariance needs to be computed~instead~of~$J$~\mbox{different}~ones.

    Next, we illustrate the importance of the correction term $\nabla_u\cdot\mathbf P_{U,t}$ in the localized CBS dynamics~\cref{eq:lcbs:loc-prec:lcbs} on the sampling problem~\cref{eq:num:multi:diffpeaks}. We consider the preconditioner $\mathcal C_{\alpha, \alpha\lambda\mathcal C(\rho_t)}(\rho_t; U_t)$ with $\lambda=0.5$ and $\alpha\to0$ and compare the normal localized CBS algorithm to a modified version where the correction term is omitted. We also apply polarized CBS, which does not have a similar correction term, with $\alpha=10$ and~$\lambda\in\{0.001, 0.002, 0.005\}$. We use $N=1000$ timesteps,~${J=200}$ particles, and an initial covariance $\Sigma_0 = 2I$ for all samplers. \Cref{fig:num:multi:corr} illustrates that only localized CBS with the correction term manages to accurately retrieve $\pi$. As before, polarized CBS only samples from a distribution approximating $\pi$.

    \begin{figure}
        \centering
        \includegraphics[width=.8\textwidth]{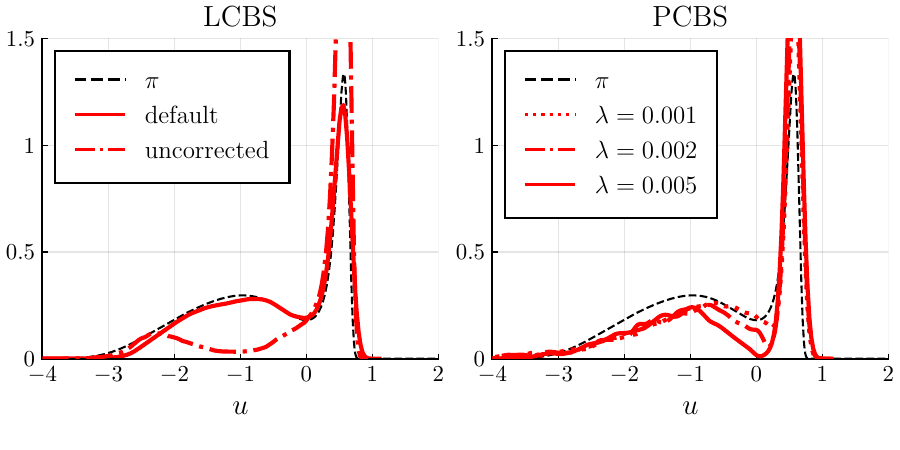}
        \caption{Localized and polarized CBS samples for~\cref{eq:num:multi:diffpeaks}}
        \label{fig:num:multi:corr}
    \end{figure}

    \subsection{A tent-shaped distribution} \label{sec:num:tent}
    In the next example we study convergence of localized CBS for the one-dimensional target distribution corresponding~to
    \begin{equation} \label{eq:num:tent:tent}
        \pi(u) = \widehat\pi(u) = \max(0, 1 - \abs u),
    \end{equation}
    a tent-shaped distribution with finite support. Our goal is not to compare to other methods, but rather to gain insight into the choice of $(\beta, \kappa)$. We will answer the question why, even though choosing~${\beta\rightarrow\infty}$ and $\kappa\rightarrow0$ seems desirable, it may still make sense to choose moderate values for both parameters in practice. \Cref{fig:num:tent:fig} shows the Wasserstein-2 distance between the exact distribution $\pi$ and the empirical distribution of localized CBS ensembles with a given number~$J$ of particles. This is compared for several $(\beta, \kappa)$ combinations. In this subsection, we find the empirical distribution for any $(\beta, \kappa, J)$ by simulating localized CBS $480$ times and, for each simulation, using the positions of $50$ randomly chosen particles at the~final~time~step~as~\mbox{samples}.

    \begin{figure}
        \centering
        \includegraphics[width=.65\textwidth]{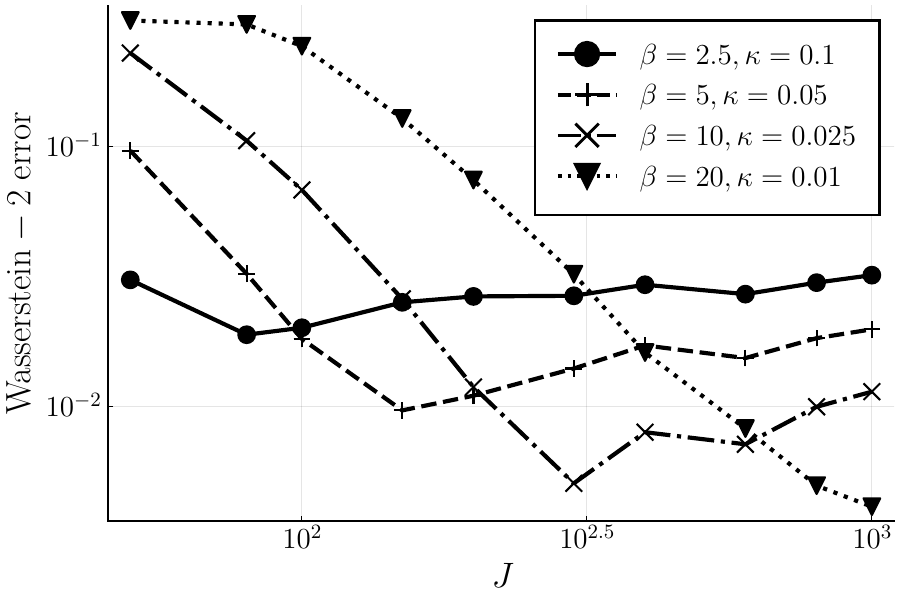}
        \caption{Wasserstein-$2$ error between the localized CBS samples and the target distribution~\cref{eq:num:tent:tent}}
        \label{fig:num:tent:fig}
    \end{figure}

    After a certain value, increasing $J$ no longer improves accuracy. This is caused by the discrepancy between the mean-field distribution and $\pi$, which decreases with large $\beta$ and small $\kappa$. However, these more accurate mean-field dynamics become increasingly difficult to approximate with a finite number of particles, as fewer particles get to contribute substantially to the weighted means. Hence, when large numbers of particles are computationally feasible, a high accuracy can be obtained by choosing~$\beta\gg0$ and $\kappa\approx0$. When relatively few particles can be used, on the other hand, the accuracy is actually highest for more moderate~$\beta$~and~$\kappa$~choices.
    \begin{remark}
        As $J$ keeps increasing, the errors in \cref{fig:num:tent:fig} become slightly larger again. This may be caused by a combination of two factors: (i) the Moreau envelope in \cref{sec:cbs-approx} approximates $\pi$ by a distribution with longer tails and (ii) when $J$ is small, localized CBS tends to undersample the tails. While the first factor is a source of error present for all $J$, the second factor may partially offset it for small $J$, resulting in an~error~smaller~than~that~for~${J\rightarrow\infty}$.
    \end{remark}

    \subsection{A Darcy flow inverse problem}
    Our final example is meant to assess how localized CBS scales to higher dimensions. We choose a more complex model than in earlier experiments, but the resulting posterior is near-Gaussian and does not require localization. We mainly use this example since it is slightly more realistic than the previous ones, and its dimension can be varied straightforwardly. A near-Gaussian distribution is a best-case scenario: scaling well here is likely a prerequisite for scaling well for non-Gaussian distributions.

    Consider the spatial domain~${\Omega\coloneqq[0, 1]^2}$. The Darcy flow inverse problem is to recover parameters $u$ of the permeability field $a(x, u)$ from noisy measurements of the pressure field $p(x)$ satisfying\hspace{-.1cm}
    \begin{subequations} \label{eq:num:darcy:darcy}
    \begin{align}
        -\nabla\cdot\bigl(a(x, u) \nabla p(x)\bigr) &= f(x), &x&\in\Omega,\\
        p(x) &= 0, &x&\in\partial\Omega.
    \end{align}
    \end{subequations}
    Here, $f(x) = c$ is a fluid source that we set to a constant value.

    The permeability $a(x, u)$ is modeled as a log-normal random field with mean zero and covariance operator $(-\Delta + \tau^2)^{-s}$ of $\log a$, where $-\Delta$ represents the Laplacian with homogeneous Neumann boundary conditions. This corresponds to a standard normal prior on the parameters $u_k$ in the Karhunen--Lo\`eve (KL) expansion
    \begin{equation} \label{eq:num:darcy:kl}
        \log a(x, u) = \sum\nolimits_{k\in\bN^2\setminus\{(0,0)\}} u_k\sqrt{\lambda_k}\,\phi_k(x),
    \end{equation}
    with eigenpairs $\lambda_k = (\pi^2\norm{k}_2^2 + \tau^2)^{-s}$ and $\phi_k = c_k\cos(\pi k_1x_1)\cos(\pi k_2x_2)$, where~${c_k=\sqrt2}$ if~$k_1k_2=0$ and $c_k=2$ otherwise (see, e.g.,~\cite{huangIteratedKalmanMethodology2022c}). We choose parameters~$(\tau, s)=(3, 2)$ and truncate the expansion~\cref{eq:num:darcy:kl} after the $d$ terms with largest eigenvalues $\lambda_k$. The coefficients~$u_k$ of these terms are the unknown parameters~$u$ with prior $\pi_\mathrm{prior} = \mathcal N(0, \Gamma_0)$, with $\Gamma_0 = I$, in the Bayesian inverse problem~\cref{eq:intro:bip:bip}. The data~$y$ consists of measurements of~$p$, taken at $49$ equispaced points and subject to noise known to be distributed according to $\pi_\mathrm{noise} = \mathcal N(0, \Gamma)$ with $\Gamma = 10^{-4}I$. The model $G$ solves~\cref{eq:num:darcy:darcy} with central finite differences on a grid with stepsize $h = 2^{-5}$. An observation is generated by sampling $\{u_k\}_k$ from the prior, solving~\cref{eq:num:darcy:darcy} with stepsize~$h=2^{-9}$, and perturbing that solution with a random sample from $\pi_\mathrm{noise}$. As described in \cref{sec:intro:bip}, this Bayesian inverse problem corresponds to~\cref{eq:intro:intro:dist} with
    \begin{equation}
        V(u) = \frac12\norm{y - G(u)}_{\Gamma}^2 + \frac12\norm u_{\Gamma_0}^2.
    \end{equation}

    \begin{figure}
        \centering
        \includegraphics[width=.8\textwidth]{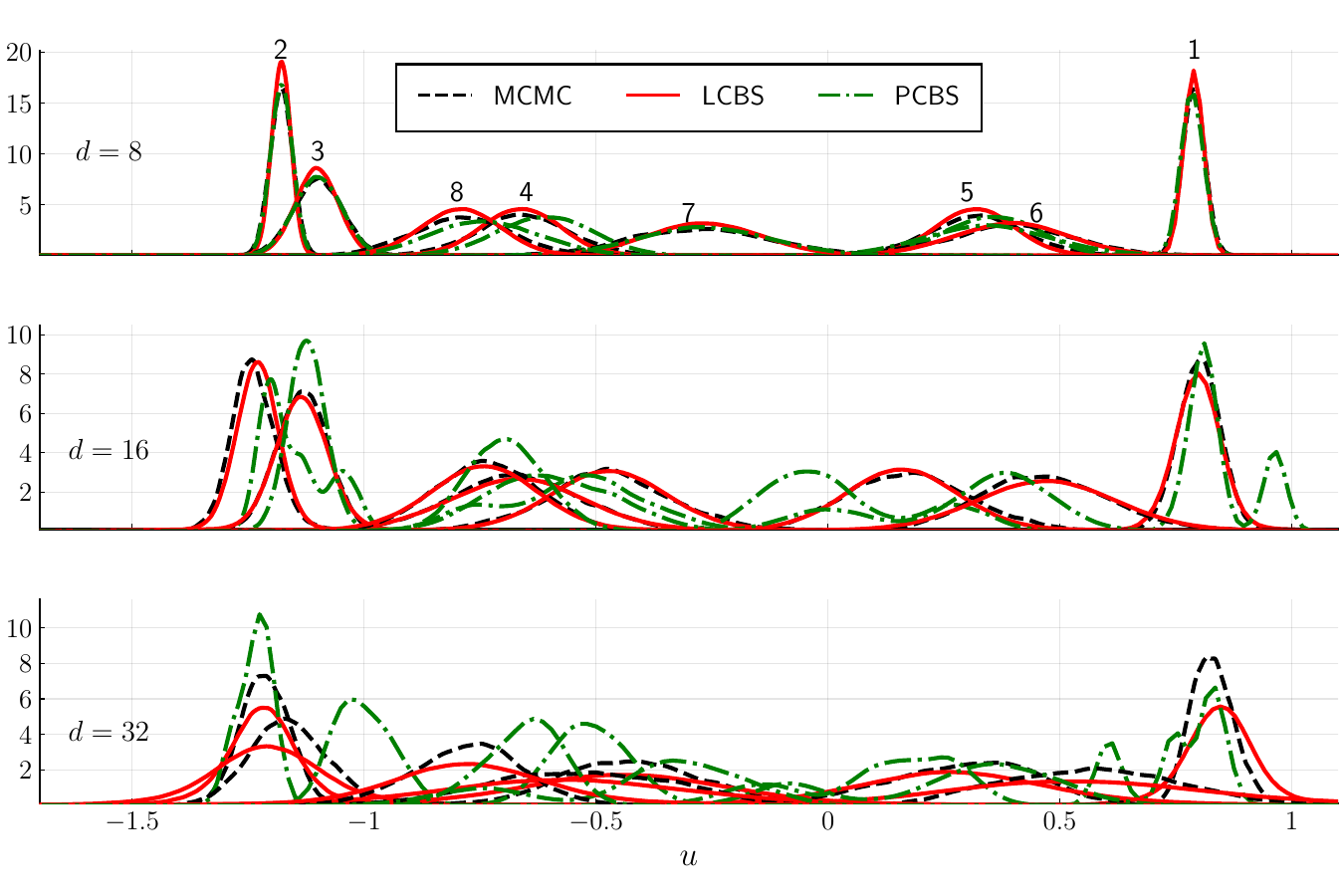}
        \caption{MCMC, localized CBS, and polarized CBS samples of the first $8$ KL terms in the $d$-dimensional Darcy flow problem, with $d\in\{8, 16, 32\}$. The $d=8$ plot labels the modes in order of decreasing eigenvalue $\lambda_k$.}
        \label{fig:num:darcy:fig}
    \end{figure}

    Since the resulting posterior is unknown, we need to compute a sensible reference solution. For this, we use random-walk MCMC: we run 32 chains of $10^5$ steps each, of which the first quarter is discarded as burn-in. For localized CBS, we tuned parameters to obtain good results for $d=8$: we use $N=400$ timesteps of length $\dt=0.03$,~${J=1000}$ particles, a large initial covariance $\Sigma_0 = 9I$ (similarly to~\cite{carrilloConsensusbasedSampling2022}), parameters $(\beta, \kappa) = (3, 0.2)$, and our default preconditioner $\mathcal C(\rho_t)$.

    We then test whether these parameters can be scaled up to higher dimensions. \Cref{fig:num:darcy:fig} compares the reference samples to the localized CBS samples. It shows the marginalized distributions for the first 8 modes, for the number of dimensions $d\in\{8,16,32\}$. The modes are successfully located for each $d$. When $d\in\{8,16\}$, the uncertainty around the mode is captured well; when $d=32$, it is slightly overestimated. This matches the argument in \cref{rem:comput:scale} that scaling localized CBS to high dimensions may prove difficult, although further research is needed.

    For comparison, we conduct the same experiment with polarized CBS\footnote{Localized ALDI diverges with this initial distribution.}, using $J=1000$ particles over $N=400$ timesteps of length $\dt=0.02$. We chose the parameters $D=I$ and $\lambda=0.25$ to work well for $d=8$. The results, shown in \cref{fig:num:darcy:fig} as well, display a similar trend as for localized CBS (but more pronounced): in higher dimensions, these parameters no longer succeed in capturing the target distribution. Increasing the number of particles with the dimension may help, but becomes intractable quickly as $d$ increases. As for localized CBS, further research is needed to determine whether polarized CBS can be scaled up efficiently to higher dimensions.

\section{Conclusions} \label{sec:concl}
    We have derived localized consensus-based sampling (localized CBS) as an alternative formulation of consensus-based sampling (CBS) obtained from ensemble-preconditioned Langevin dynamics through a sequence of approximations involving the Moreau envelope and a proximal operator. By retaining only the steps that remain valid beyond the Gaussian setting, we obtained a variant closely related to polarized CBS but endowed with full affine invariance. In the mean-field limit, localized CBS is exact for Gaussian targets.

    Numerical experiments demonstrate that localized CBS achieves comparable or improved performance relative to polarized CBS and localized ALDI, particularly in non-Gaussian and moderately multimodal problems. The method is fully gradient-free and easily parallelizable.

    Several open research directions remain. First, the theoretical understanding of localized CBS---including its non-Gaussian convergence, well-posedness, and mean-field limit---is still incomplete. As a first step, the ODEs~\cref{eq:lmm:lcbs:gauss:moments} may allow deriving an explicit convergence rate and comparing different $(\alpha, \lambda)$ choices in the Gaussian case. Second, as discussed in \cref{rem:comput:scale}, the efficiency of localized CBS in high-dimensional settings has not been analyzed and likely depends on problem structure and localization scale. Third, localized CBS may be suitable for sampling from non-differentiable potentials, similarly to other prox-based samplers~\cite{leeStructuredLogconcaveSampling2021,pereyraProximalMarkovChain2016,titsiasAuxiliaryGradientBasedSampling2018}. Finally, adapting localized CBS to exploit additional problem information---for example, hybrid strategies that use gradient evaluations only for selected particles---may help extend its applicability to more complex sampling tasks.

\section*{Acknowledgments}
    We are grateful to Kathrin Hellmuth for her comments on an earlier draft of this paper. We also wish to thank the anonymous referees for providing valuable feedback and suggestions, which greatly improved this paper. Our work was partially funded by the Research Foundation~--~Flanders (FWO) (grants 1169725N, G081222N, G033822N, and G0A0920N) and the KU Leuven Research Council (C1 projects C14/23/098 and C14/24/103).

\appendix

\section{Various technical proofs} \label{sec:apdx}
    \begin{proof}[Proof of~\cref{eq:lmm:lcbs:gauss:moments:C,eq:lmm:lcbs:gauss:moments:mu-w}]
        Consider means $(m, m_t)$ and covariances $(\Sigma, \Sigma_t)$ such that~${\pi\sim\mathcal N(m, \Sigma)}$ and $\rho_t\sim\mathcal N(m_t, \Sigma_t)$. We prove~\cref{eq:lmm:lcbs:gauss:moments:mu-w} similarly to~\cite{bungertPolarizedConsensusbasedDynamics2024}, by using the formula found in, e.g.,~\cite[\S8.1.8]{petersenMatrixCookbook2006} to compute $\mu_{\beta,\kappa\mathbf P_{*,t}}(\rho_t; U_t)$ as
        \begin{align*}
            &\frac{\int v\exp(-\frac\beta{2\kappa}\norm{v - U_t}_{\mathbf P_{*,t}}^2)\exp(-\frac\beta2\norm{v - m}_\Sigma^2)\exp(-\frac12\norm{v - m_t}_{\Sigma_t}^2)\,\mathrm dv}{\int\exp(-\frac\beta{2\kappa}\norm{v - U_t}_{\mathbf P_{*,t}}^2)\exp(-\frac\beta2\norm{v - m}_\Sigma^2)\exp(-\frac12\norm{v - m_t}_{\Sigma_t}^2)\,\mathrm dv}\\
            &\quad= \Bigl(\beta\Sigma^{-1} + \Sigma_t^{-1} + \frac\beta\kappa\mathbf P_{*,t}^{-1}\Bigr)^{-1}\Bigl(\beta\Sigma^{-1}m + \Sigma_t^{-1}m_t + \frac\beta\kappa\mathbf P_{*,t}^{-1}U_t\Bigr).
        \end{align*}
    \Cref{eq:lmm:lcbs:gauss:moments:C:K,eq:lmm:lcbs:gauss:moments:C:C} follow directly from the definitions. \Cref{eq:lmm:lcbs:gauss:moments:C:Calpha,eq:lmm:lcbs:gauss:moments:C:Cw} are derived similarly to $\mu_{\beta,\kappa\mathbf P_{*,t}}(\rho_t; U_t)$ above.
    \end{proof}

    \begin{proof}[Proof of \cref{lmm:comput:corr}]
        In~\cite[Lemma 3.11]{reichFokkerPlanckParticleSystems2021a}, it is shown that with general weights,
        \begin{equation}
            \nabla_{u^i}\cdot\widehat{\mathbf P} = \omega^{ii}\,(d+1)(U^i - \bar U^{(i)}) + \sum\nolimits_{j=1}^J\left(U_c^{(i),j}(U_c^{(i),j})^T - \bar U^{(i)}(\bar U^{(i)})^T\right)\nabla_{u^i}\omega^{ij}.\hspace{-.1cm}
        \end{equation}
        With our weight function, we find that
        \begin{align*}
            &\nabla_{u^i}\omega^{ij} = \nabla_{u^i}\Bigl(w^{ij}\Big/\sum\nolimits_{k=1}^Jw^{ik}\Bigr) = \Bigl(\nabla_{u^i} w^{ij} - \omega^{ij}\nabla_{u^i} \Bigl(\sum\nolimits_{k=1}^Jw^{ik}\Bigr)\Bigr) \Big/ \sum\nolimits_{k=1}^Jw^{ik}\\
            &\quad= \Bigl[-\frac1{2\lambda}\left(2\widehat{\mathcal C}^{-1}(U^i - U^j) - X^{ij}\right) + \delta^{ij}\alpha\frac{\nabla_{u^i}\widehat\pi(U^i)}{\widehat\pi(U^i)}\Bigr]\omega^{ij}\\
            &\quad\qquad+ \omega^{ij}\sum\nolimits_{k=1}^J\Bigl[\frac1{2\lambda}\left(2\widehat{\mathcal C}^{-1}(U^i - U^k) - X^{ik}\right) - \delta^{ik}\alpha\frac{\nabla_{u^i}\widehat\pi(U^i)}{\widehat\pi(U^i)}\Bigr]\omega^{ik}\\
            &\quad= \frac{\omega^{ij}}\lambda\Bigl(\bigl[\widehat{\mathcal C}^{-1}(U^j - U^i) + \frac{X^{ij}}2\bigr] + \bigl[\widehat{\mathcal C}^{-1}U_c^{(i),i} - \sum\nolimits_{k=1}^J \omega^{ik}\frac{X^{ik}}2 \bigr]\Bigr) + (\delta^{ij}-\omega^{ij})\alpha\omega^{ii}\frac{\nabla_{u^i}\widehat\pi(U^i)}{\widehat\pi(U^i)}\\
            &\quad= \frac{\omega^{ij}}\lambda\Bigl(\widehat{\mathcal C}^{-1}U_c^{(i),j} - \sum\nolimits_{k=1}^J \omega^{ik}\frac{X^{ik}}2 + \frac{X^{ij}}2\Bigr) + (\delta^{ij}-\omega^{ij})\alpha\omega^{ii}\frac{\nabla_{u^i}\widehat\pi(U^i)}{\widehat\pi(U^i)},
        \end{align*}
        with $(X^{ik})_\ell=(U^i - U^k)^T\widehat{\mathcal C}^{-1}\,(\partial_{u^i_\ell} \widehat{\mathcal C})\,\widehat{\mathcal C}^{-1}(U^i - U^k)$. We have $\partial_{u^i_\ell} \widehat{\mathcal C} = J^{-1}\bigl(U_c^ie_\ell^T+e_\ell(U_c^i)^T\bigr)$ and%
        \begin{subequations}
        \begin{align}
            X^{ij}/2 &= J^{-1}\bigl((U_c^i)^T\widehat{\mathcal C}^{-1}(U^i - U^j)\bigr)\widehat{\mathcal C}^{-1}(U^i - U^j),\\
            \sum\nolimits_{k=1}^JX^{ik}\omega^{ik}/2 &= J^{-1}\widehat{\mathcal C}^{-1}\bigl(U_c^{(i),i}(U_c^{(i),i})^T + \widehat{\mathbf P}\bigr)\widehat{\mathcal C}^{-1}U_c^i.
        \end{align}
        \end{subequations}
        We split $\nabla_{u^i}\cdot\widehat{\mathbf P} = A + (B_1 + B_2 + B_3 - C_1 - C_2 - C_3)/\lambda + D$ with
        \begin{align*}
            A &\coloneqq \omega^{ii}(d+1)(U^i-\bar U^{(i)}),\\
            B_1 &\coloneqq \sum\nolimits_{j=1}^JU_c^{(i),j}(U_c^{(i),j})^T\omega^{ij}\bigl[\widehat{\mathcal C}^{-1}U_c^{(i),j}\bigr] = \bm U_c^{(i)}\bigl(\omega^i \odot \mathrm{diag}((\bm U_c^{(i)})^T \widehat{\mathcal C}^{-1}\bm U_c^{(i)})\bigr),\\
            B_2 &\coloneqq \sum\nolimits_{j=1}^JU_c^{(i),j}(U_c^{(i),j})^T\omega^{ij}\bigl[-J^{-1}\widehat{\mathcal C}^{-1}\bigl(U_c^{(i),i}(U_c^{(i),i})^T + \widehat{\mathbf P}\bigr)\widehat{\mathcal C}^{-1}U_c^i\bigr]\\
            &= -J^{-1}\widehat{\mathbf P}\widehat{\mathcal C}^{-1}(U_c^{(i),i}(U_c^{(i),i})^T + \widehat{\mathbf P})\widehat{\mathcal C}^{-1}U_c^i,\\
            B_3 &\coloneqq \sum\nolimits_{j=1}^JU_c^{(i),j}(U_c^{(i),j})^T\omega^{ij}\bigl[J^{-1}\bigl((U_c^i)^T\widehat{\mathcal C}^{-1}(U^i - U^j)\bigr)\widehat{\mathcal C}^{-1}(U^i - U^j)\bigr]\\
            &= J^{-1}\bm U_c^{(i)} (\omega^i \odot \mathrm{diag}((\bm U_c^{(i)})^T\widehat{\mathcal C}^{-1}(\bm U - U^i\bm 1^T)) \odot ((U_c^i)^T \widehat{\mathcal C}^{-1}(\bm U - U^i\bm 1^T))^T),\\
            C_1 &\coloneqq \sum\nolimits_{j=1}^J \bar U^{(i)}(\bar U^{(i)})^T\omega^{ij}\bigl[\widehat{\mathcal C}^{-1}U_c^{(i),j}\bigr]= \bar U^{(i)}(\bar U^{(i)})^T\widehat{\mathcal C}^{-1}\Bigl(\sum\nolimits_{j=1}^J\omega^{ij}U_c^{(i),j}\Bigr) = 0,\\
            C_2 &\coloneqq \sum\nolimits_{j=1}^J \bar U^{(i)}(\bar U^{(i)})^T\omega^{ij}\Bigl[- \sum\nolimits_{k=1}^J \omega^{ik}\frac{X^{ik}}2\Bigr] = -\bar U^{(i)}(\bar U^{(i)})^T\sum\nolimits_{k=1}^J \omega^{ik}\frac{X^{ik}}2,\\
            C_3 &\coloneqq  \sum\nolimits_{j=1}^J \bar U^{(i)}(\bar U^{(i)})^T\omega^{ij}\Bigl[\frac{X^{ij}}2\Bigr] = \bar U^{(i)}(\bar U^{(i)})^T\sum\nolimits_{j=1}^J \omega^{ij}\frac{X^{ij}}2 = -C_2,\\
            D &\coloneqq \sum\nolimits_{j=1}^J \left(U_c^{(i),j}(U_c^{(i),j})^T - \bar U^{(i)}(\bar U^{(i)})^T\right)(\delta^{ij}-\omega^{ij})\alpha\omega^{ii}\frac{\nabla_{u^i}\widehat\pi(U^i)}{\widehat\pi(U^i)}\\
            &= \alpha\omega^{ii}\left(U_c^{(i),i}(U_c^{(i),i})^T - \widehat{\mathbf P}\right)\frac{\nabla_{u^i}\widehat\pi(U^i)}{\widehat\pi(U^i)}. \qedhere
        \end{align*}
    \end{proof}

\bibliography{references}

\end{document}